\documentclass[12pt]{elsarticle}
\usepackage{makecell} % Provides \Xhline
\usepackage{hhline}   % Provides \hhline
\usepackage{graphicx}
\usepackage{subcaption}
\usepackage{booktabs}
\usepackage{hhline}
\usepackage{epstopdf}
\usepackage{color}
\usepackage{latexsym}
\usepackage{setspace}
\usepackage{amsmath}
\usepackage{amssymb}
\usepackage[numbers]{natbib}
\usepackage{hyperref}
\usepackage{multirow}
\usepackage{multirow}
\hypersetup{
	colorlinks=true,
	linkcolor=blue,
	filecolor=blue,
	urlcolor=blue,
}

\makeatletter
\def\ps@pprintTitle{%
	\let\@oddhead\@empty
	\let\@evenhead\@empty
	\let\@oddfoot\@empty
	\let\@evenfoot\@oddfoot
}
\makeatother
\makeatletter \oddsidemargin  -.5cm \evensidemargin -.5cm \textwidth
17.5cm \topmargin -.65in \textheight 24cm

\newtheorem{defn}{Definition}[section]
\newcommand{\bd}{\begin{definition}}
	\newcommand{\ed}{\end{definition}}
\newcommand{\br}{\begin{remark}}
	\newcommand{\er}{\end{remark}}
\newcommand{\bea}{\begin{eqnarray}}
	\newcommand{\eea}{\end{eqnarray}}
\newcommand{\beann}{\begin{eqnarray*}}
	\newcommand{\eeann}{\end{eqnarray*}}
\newtheorem{theorem}{Theorem}[section]

\newtheorem{lemma}[theorem]{Lemma}

\newtheorem{corollary}[theorem]{Corollary}
\newtheorem{remark}{Remark}[section]
\newtheorem{example}{Example}[section]

\numberwithin{equation}{section} \numberwithin{equation}{section}
\title{Estimating ratio of variance components in a balanced one-way random effect model}
\author{Somnath Mondal and Lakshmi Kanta Patra\footnote
	{\baselineskip=10pt
		~lkpatra@iitbhilai.ac.in; patralakshmi@gmail.com}  \\
	Department of Mathematics\\
	Indian Institute of Technology Bhilai, Durg, India-491002}
\begin{document}
	\begin{frontmatter}
		\date{}
		%\maketitle
		\begin{abstract}
			This paper considers the problem of estimating the ratio of variance components of a balanced one-way random effects model. Existing estimators of ratio of variance, such as the ANOVA and maximum likelihood estimators, can take negative values and are inadmissible under quadratic loss. We therefore study the problem under a general class of scale-invariant loss functions. We characterize the class of scale-equivariant estimators of ratio of variance and obtain the best equivariant estimator. We then derive a class of estimators that dominate the BEE and extend this class to obtain further improved estimators. We propose a smooth improved estimator of ratio of variance. For interval estimation of ratio of variance, we derive an asymptotic confidence interval, bootstrap confidence intervals, a generalized confidence interval and a highest posterior density credible interval. The performance of the proposed improved point estimators and confidence interval procedures is assessed through an extensive simulation study. Finally, we provide a real data analysis for implementation purposes.\\		
			
		\end{abstract}
	\end{frontmatter}
	\section{Introduction} 
	
	The balanced one-way random effects model is widely used in scientific and industrial investigations. It applies when the observations arranged in groups, the groups are selected at random, and each group has the same number of observations. The model decomposes the total variation in the data into two parts: variation between groups and variation within groups. Many practical problems fit naturally into this setting. For example, in an agricultural field trial, farms may be selected at random and the same number of plots observed on each farm. The variability between farms come from environmental factors such as soil fertility and climate. Variability between plots on the same farm come from measurement error and small local changes in conditions. Similarly in manufacturing, production batches may be selected at random and the same number of items inspected from each batch. The between-batch variation shows how consistent the production process is. The within-batch variation reflects random production and measurement errors. A similar situation arises in quality assessment, where several measurements are taken on specimens from randomly chosen laboratories or production centres. In all these situations, the balanced one-way random effects model provides a natural framework for separating the two sources of variation. In this work, we consider the balanced one-way random effects model
	\begin{align}\label{main}
		y_{ij} = \mu + a_i + e_{ij}, \quad i = 1, \dots, n_1;\ j = 1, \dots, n_2,
	\end{align}
	where the $a_i$ are independent $N(0, \sigma_1^2)$ random variables, the $e_{ij}$ are independent $N(0, \sigma_2^2)$ random variables, and the $a_i$ and $e_{ij}$ are mutually independent. Here $-\infty < \mu < \infty$, $\sigma_1^2 \geq 0$ and $\sigma_2^2 > 0$. The parameter $\sigma_1^2$ is the between-group variance component and $\sigma_2^2$ is the within-group (error) variance component. Our aim is to derive a class of improved estimators of the ratio $\sigma_1^2/\sigma_2^2$. This model is one of the most widely used models in variance components analysis. It is used whenever data are collected from several groups and the observations vary both between groups and within each group. It arising in various applications include agriculture, industrial quality control, biological experiments, educational testing and reliability studies. The aim of this work is to derive improved estimators of the ratio $\sigma_1^2/\sigma_2^2$.
	
	The estimation of variance components in the one-way random effects model has been studied extensively in the literature. For example, \cite{klotz1969mean} and \cite{portnoy1971formal} studied the estimation of the variance components $\sigma_1^2$ and $\sigma_2^2$ under the mean squared error (MSE) criterion. However, comparatively little attention has been paid to the point estimation of the ratio $\sigma_r=\sigma_1^2/\sigma_2^2$ or equivalently $\lambda = 1 + n_2\sigma_r = E(\mbox{MSA})/E(\mbox{MSE})$ where $\mathrm{MSA}$ and $\mathrm{MSE}$ denote the between-group and within-group mean squares respectively. This ratio is of considerable practical importance. It is a one-to-one function of the intraclass correlation coefficient $\mbox{Corr}(y_{ij},y_{ij'})= \rho = \sigma_r/(1+\sigma_r)$, which measures the similarity of observations within the same group. Moreover, since the variance of the overall mean, $\operatorname{Var}(\bar{y}) = \sigma_2^2(1+n_2\sigma_r)/(n_1n_2)$, depends on the variance components essentially through $\sigma_r$, this ratio determines the optimal choice of $n_1$ and $n_2$ when planning future experiments (see \cite{scheffe1959analysis}, \cite{snedecor1980statistical}). Latter, \cite{loh1986improved} studied the estimation of the ratio of variance components in a one-way random effects model. He established the inadmissibility of the maximum likelihood (ML), restricted maximum likelihood (REML) and Bayes estimators of $\theta$ by constructing an improved optimal estimator, denoted by $\theta_{\mathrm{opt}}$. Motivated by this work \cite{das1990inadmissibility} studied the estimation of the ratio of variance components under squared error loss and proved the inadmissibility of the optimal estimator $\theta_{\mathrm{opt}}$ obtained by \cite{loh1986improved}. One of the main drawbacks of the usual estimators of $\sigma_r$ is that they can take negative values. \cite{das1992improved} showed that both the ANOVA estimator and the unbiased estimator of $\sigma_r^2$ can be negative, and obtained non-negative estimators that dominate the usual estimators under the mean squared error loss function. There are less work has been done using decision theoretic approach for this direction. \cite{srivastava1999improved} proposed non-negative definite estimators of the between- and within-group covariance components in the multivariate one-way random-effects model and showed that they dominate the best unbiased estimators under Stein's loss. \cite{kubokawa2006estimation} used the Stein–Haff identity to propose truncated estimators of covariance matrices and multivariate variance components in fixed- and mixed-effects linear models, and showed that they improve upon the usual unbiased estimators under suitable loss functions. \cite{kubokawa2007estimation} considered the simultaneous estimation of the regression coefficients and the error variance in a linear regression model under the Kullback–Leibler loss. They proposed truncated and shrinkage estimators that dominate the usual estimators and used them to derive AIC variants for model selection. For further studies in this direction under various settings, we refer to \cite{heine1993nonnegative},  \cite{kubokawa1995estimation}, \cite{kubokawa2000bayes}, \cite{wei2004empirical}, \cite{bobotas2017estimation}.

	Decision-theoretic methods have also been used extensively in recent years to estimate scale parameters and their functions in various probability models. We now briefly review some recent work in this direction. \cite{kayal2024estimating} studied the estimation of the scale parameters of several exponential distributions with ordered restriction under the linex loss function. They obtained a sufficient condition under which the equivariant class of estimators improves upon the usual estimators.  Specifically they obtained \cite{stein1964}-type and \cite{brewster1974improving}-type improved estimator. \cite{mondal2025estimating} studied the estimation of the ordered scale parameters of two scale mixtures of exponential distributions under the Stein and symmetric loss functions. They have shown the inadmissibility of best affine equivariant estimator and proposed a class of dominating estimators. Also they showed that the \cite{brewster1974improving}-type estimator is generalized Bayes estimator. \cite{mondal2026estimating} considered the estimation of positive power of ordered standard deviations of two normal population under a general loss function. They have used the techniques of \cite{stein1964}, \cite{kubokawa1994unified} and \cite{brewster1974improving} to find the improved estimators under a general scale invariant loss function. 
	
	Following this line of research, we study the estimation of the ratio of variance components in the balanced one-way random effects model (\ref{main}) from a decision-theoretic point of view. This estimation problem is studied under a general class of bowl-shaped loss functions $L(t)$. This loss function $L(t)$ must satisfy the following conditions:
	\begin{itemize}
		\item [(A1)] $L(t)$ is strictly convex with $L(1)=0$.
		\item [(A2)] The integrals involving $L(t)$ are finite and can be differentiated under the integral sign.
	\end{itemize}
	Since $\sigma_r = (\lambda - 1)/n_2$, then deriving an improved estimator of $\sigma_r$ is equivalent to finding an improved estimator of $\lambda$. Hence, our problem reduces to finding an improved estimator of $\lambda$ under the loss function $L(t)$. Our goal is to obtain estimators that dominate best equivariant estimator of $\lambda$. In particular, we proposed \cite{stein1964}-type and \cite{kubokawa1994double} IERD type estimator. Improved estimators have been obtained by many authors using these methods. For important applications, the reader may refer to \cite{misra2005estimation}, \cite{kayal2015estimating}, \cite{petropoulos2017estimation}, \cite{patra2021minimax}, \cite{patra2025inadmissibility} and the references cited therein.
	
	In addition to point estimation, we study interval estimation of $\sigma_r$ using four approaches: the asymptotic confidence interval, bootstrap confidence intervals, generalized confidence intervals, and the highest posterior density (HPD) credible interval based on the Markov chain Monte Carlo (MCMC) method. Interval estimation has been widely studied for various statistical models, and we briefly review some recent contributions. \cite{arendacka2005generalized} proposed generalized confidence intervals for a variance component in mixed linear models with two variance components and compared several test variables through simulation. \cite{burdick2006present} studied confidence interval estimation for one-factor random-effects models. They introduced a small-sample adjustment to improve the bootstrap intervals and discussed generalized confidence intervals based on the approach of \cite{weerahandi1993generalized}. \cite{ahmed2020inference} studied the Chen distribution under progressive Type-II censoring in the presence of competing risks and constructed asymptotic and bootstrap confidence intervals for the unknown parameters. \cite{ghanbari2022estimation} studied interval estimation of the stress–strength reliability $\mathcal{R}$ for Marshall–Olkin distributions based on progressively Type-II censored samples. They obtained an asymptotic confidence interval using the delta method, percentile and $t$-bootstrap confidence intervals, and an HPD credible interval based on MCMC samples. \cite{singhasomboon2023asymptotic} studied the confidence interval estimation for the ratio of the modes of two log-normal distributions. They obtained an asymptotic confidence interval along with three bootstrap-based intervals, namely, bootstrap-$t$, percentile bootstrap, and student's t bootstrap intervals. For further studies, the reader may refer to \cite{liu2016confidence} \cite{khatun2022hypothesis}, \cite{jena2023point}, \cite{sarkar2024estimating}, \cite{khatun2024inference} and the references cited therein.

	% Now observe that the model (\ref{main}) can be reduced to a canonical form by suitable choice of orthogonal transformations (see Appendix of \cite{das1990inadmissibility}). 
	
	We now reduce the model to a convenient form. By a suitable orthogonal transformation (see the Appendix of \cite{das1990inadmissibility}), model (\ref{main}) can be written in canonical form. This yields a new observation vector of length $n_1n_2$ observation vector $Z=(Z_{11},Z_{21},\dots,Z_{n_11},Z_{ij},\dots;j>1)'$ is such that 
	\begin{align}
		E(Z)=\lambda^{*}=(\sqrt{n_1n_2}\mu,0,\dots,0)',\\
		\mbox{Disp}(Z)=\Sigma=\mbox{Diag}(\sigma_2^2\lambda I_{n_1},\sigma_2^2I_{n_2-n_1}).
	\end{align}
	We consider the problem of estimating $\lambda$ with respect to the general loss function $L(t)$. The problem of estimating $\lambda$ remains invariant under the following group actions $\mathcal{G}$:
	\begin{equation}
		\begin{aligned}\label{2.3}
			Z &\rightarrow DO(Z+b), \quad 
			D= \begin{pmatrix} adI_{n_1} & 0 \\ 0 & aI_{n_2-n_1} \end{pmatrix}, \quad 
			b=(b_{11},0,\dots,0)', \quad a,d \neq 0,\\
			\Sigma &\rightarrow D\Sigma D, \quad 
			\lambda^* \rightarrow DO(\lambda + b), \quad 
			O \in \mathcal{O}= \text{the set of orthogonal matrices}.
		\end{aligned}
	\end{equation}
	Then the equivariant estimator $\delta$ would be such that 
	\begin{equation}\label{2.4}
		\delta(DO(Z+b))=d^2 \delta(Z).
	\end{equation}
	Take $a=d=1$, that is $D=I$, we see that 
	\begin{equation*}
		\delta(O(Z+b))= \delta(Z)
	\end{equation*}
	for all orthogonal $O\in \mathcal{O}$. This implies the $\delta(Z)$ depends on $Z$ only through $Z_{11},$ $S_1$ and $S_2$ that is 
	\begin{equation}\label{2.5}
		\delta(Z)=\phi((Z_{11}+b_{11}), S_1,S_2)
	\end{equation}
	where $S_1=\sum_{i=1}^{n_1}Z_{ij}^2=n_2\sum_{i=1}^{n_1}(\bar{y}_i-\bar{y})^2$, $S_2=\sum_{i=1}^{n_1}\sum_{j=2}^{n_2}Z_{ij}^2=\sum_{i=1}^{n_1}\sum_{j=1}^{n_2}(y_{ij}-\bar{y}_i)^2$.
	Thus from (\ref{2.4}) using (\ref{2.5}), we get
	\begin{equation}
		\phi(ad(Z_{11}+b_{11}),a^2d^2S_1,a^2S_2)=d^2\phi((Z_{11}+b_{11}), S_1,S_2).
	\end{equation}
	Hence setting $b_{11}=-Z_{11}$, $a^2=1/S_2$, $d^2=S_2/S_1$, we obtain
	\begin{equation}\label{be}
		\delta(Z)=\phi(S_1,S_2)=c\frac{S_1}{S_2},~~ \mbox{where c is a positive real number.}
	\end{equation}
	Thus any $\mathcal{G}$-equivariant estimator $\phi$ would be of the form $cS_1/S_2$ where $c$ is some positive constant. Note that $S_1$ and $S_2$ are independent. Furthermore $S_1\sim\sigma_2^2\lambda\chi^2_{p_1}, ~ S_2\sim\sigma_2^2\chi^2_{p_2} ~~ \mbox{with} ~ p_1=n_1-1, ~ p_2=(n_2-1)n_1.$ The best equivariant estimator (BEE) of $\lambda$ under the loss function $L(t)$ is obtained from the following lemma.
	\begin{lemma}
		Under the a scale invariant loss function $L(t)$, the BEE of $\lambda$ is given by $\delta_0=c_0\frac{S_1}{S_2}$ where $c_0$ is the unique solution of 
		\begin{equation}\label{bae}
			EL'(c_0V)V=0
		\end{equation}
		where $V=\frac{S_1}{\lambda S_2}\sim \frac{p_1}{p_2}F_{p_1,p_2}$ and $L'(\cdot)$ denotes the derivative of the loss function $L(\cdot)$.
	\end{lemma}
	\begin{example}
		\begin{itemize}
			\item [(i)]	Let $L_1(t)=(t-1)^2$. The BEE of $\lambda$ is obtain as $\delta_{01}=\frac{p_2-4}{p_1+2}\frac{S_1}{S_2}$ provide $p_2>4$.
			\item [(ii)] Let $L_2(t)=t-\ln t -1$. Then BEE of $\lambda$ is obtained as $\delta_{02}=\frac{p_2-2}{p_1}\frac{S_1}{S_2}$ provided $p_2>2$.
			\item [(iii)] Let $L_3(t)=t+1/t -2$. Then BEE of $\lambda$ is obtained as $\delta_{03}=\left(\frac{p_2(p_2-2)}{p_1(p_1-2)}\right)^{1/2}\frac{S_1}{S_2}$ provided $p_1,p_2>2$.
		\end{itemize}
	\end{example}
	\begin{remark}
		Note that the BEE of $\lambda$ under quadratic loss function $L_1(t)$ is coincides with the the result obtained in Section $2$ of \cite{das1990inadmissibility}.
	\end{remark}
	\begin{remark}
		The UMVUE of $\lambda$ is $\delta_{U}=\frac{p_2-2}{p_1}\frac{S_1}{S_2}$ which is coincide with the BEE of $\delta_{U}$ of $\lambda$ under the entropy loss function $L_2(t)$ and hence it belongs to the class $(\ref{be})$. 
	\end{remark}
	
	The rest of the paper is organized as follows. In Section \ref{se3}, we obtain a class of improved estimators that dominate the BEE. In Section \ref{se4}, we present a smooth improved estimator of $\lambda$. To compare the performance of the obtained improved estimator of $\lambda$, we conducted an extensive simulation study in Section \ref{sim1}. In Section \ref{se5}, we obtain asymptotic confidence intervals (ACIs), bootstrap confidence intervals, generalized confidence intervals, and highest posterior density (HPD) intervals for the ratio $\sigma_r$. A simulation study to evaluate the performance of these confidence intervals is presented in Section \ref{se6}. Finally, Section \ref{da} presents a real-data analyses to illustrate the derived results.

	\section{Improved estimation of $\lambda$}\label{se3}
	In this section we will derived an improved estimators of $\lambda$ which dominates over BEE. To obtain the class of improved estimators we applying the \cite{stein1964} suggestion by considering estimators which are equivariant estimator with respect to the subgroup $\mathcal{H} ( \subset\mathcal{G} )$ which can be find by using the following group action $Z\rightarrow DOZ,~\Sigma \rightarrow D\Sigma D,~\lambda^*\rightarrow DO\lambda^*,$
	%	\begin{align*}
		%	Z\rightarrow DOZ, ~~~~\Sigma \rightarrow D\Sigma D, ~~~~~ \lambda^*\rightarrow DO\lambda^*,&&&&
		%	\end{align*}
	where $D$ and $\Sigma$ are given in (\ref{2.3}). Proceeding as in the case of construction of $\mathcal{G}$-equivariant estimator, here an $\mathcal{H}$ equivariant estimator would be of the form
	\begin{align}
		\delta_{\phi_1}=\varphi_1(W)\frac{S_1}{S_2}
	\end{align}
	where $\psi_1\left(\frac{S_1}{S_1+T}\right)\frac{S_1+T}{S_2}=\frac{\psi_1\left(W\right)}{W}\frac{S_1}{S_2}=\varphi_1(W)\frac{S_1}{S_2}$ with $W=\frac{S_1}{S_1+T}$, 
	\begin{align}\label{3.2}
		T=Z^2_{11}=n_1n_2\bar{y}^2\sim\sigma_2^2\lambda\chi^2_1(\eta),~~~ \eta=n_1n_2\mu^2/(\sigma_2^2\lambda).
	\end{align}
	
	The following theorem states that there an $\mathcal{H}$-equivariant estimator of $\lambda$ which is better than the BEE.
	\begin{theorem}
		Let $\alpha_1$ be the unique solution of the equation
		\begin{align}\label{3.3}
			E\left[L'\left(\alpha_1X\right)\right]=0
		\end{align} where $X\sim \text{Beta} \left(\frac{p_1+3}{2}, \frac{p_2-2}{2}\right)$. Then the risk of the estimator $\delta_{\phi_{01}(W)}$ is nowhere larger than that of the estimator $\delta_{\phi_1}$ provided $P\left(\phi_1(W)>d_0(W)\right)\neq 0$.
	\end{theorem}	
	\noindent\textbf{Proof.} By (\ref{3.2}) and the Poisson mixture representation of the noncentral chi-square distribution, conditional on an auxiliary random variable $K\sim \text{Poisson}\left(\frac{\eta}{2}\right),$	we have $T\mid K \sim \sigma_2^2\lambda\,\chi^2_{1+2K}.$	%Further, $S_2$ is independent of $(S_1,T)$, and given $K$, the ratio $ W=\frac{S_1}{S_1+T}$ is independent of \(S_1+T\). Since the risk function of $\delta_\phi$ depends only on $\eta$, we may assume, without loss of generality, that $\eta=1$. 
	Hence, the risk of the estimator $\delta_{\phi_1}$ is	
	\begin{align*}
		R(\delta_{\phi_1},\sigma_1,\sigma_2)=E^{K,W}E\left[L\left(\psi_1\left(\frac{S_1}{S_1+T}\right)\frac{S_1+T}{S_2\lambda}\right)\Big| K, W\right]=%E^{K}E\left[L\left(\frac{\psi\left(W\right)}{W}V\right)\Big| K\right]=
		E^{K,W}E\left[L\left(\phi_1(W)V\right)\Big| K, W\right].
	\end{align*} Since $S_2$ is independent of $S_1$ and $T$ and given $K$, $S_1/(S_1+T)$ is independent of $(S_1+T)$. Define
	\begin{align*}
		R_1(d,\sigma_1,\sigma_2)=E\left(L\left(d V\right)\Big|K=k, W=w\right)
	\end{align*}
	where the conditional density of $V$ given $W=w$ and $K=k$ is 
	\begin{align*}
		f_k(v)\propto v^{\frac{p_1+2k-1}{2}} (v+w)^{-\frac{p_1+p_2+2k+1}{2}}, ~~~ v>0.
	\end{align*}
	The risk function is minimized at $d_k(w)$, where $d_k(w)$ is the unique solution of 
	\begin{align*}
		E\left(L'\left(d_k(w)V\right)V\Big| K=k, W=w\right)=0.
	\end{align*}
	Using Lemma A.2. from \cite{bobotas2010estimation}, we have 
	\begin{align*}
		E\left(L'\left(d_0(w)V\right)V\Big| K=k, W=w\right)&\geq E\left(L'\left(d_0(w)V\right)V\Big| K=0,W=w\right)\\
		&=0= E\left(L'\left(d_k(w)V\right)V\Big| K=k, W=w\right).
	\end{align*}
	Consequently we get $d_k(w)\leq d_0(w)$, where $d_0(w)$ is the unique solution of 
	\begin{align*}
		E\left(L'\left(d_0(w)V\right)V\Big| K=0, W=w\right)=0.
	\end{align*}
	By taking the transformation $x=\frac{v}{w}$ on the integral equation, we get $E\left[L'\left(Xwd_0(w)\right)\right]=0$ with $X \sim \text{Beta} \left(\frac{p_1+3}{2}, \frac{p_1-2}{2}\right)$. Comparing with (\ref{3.3}), we get $\alpha_1=wd_0(w)$. Consider $\phi_{01}(w)=\min\left\{\phi_1(w),d_0(w)\right\}$, then for $P\big(d_0(w)<\phi_1(w)\big)\neq 0$ we get $d_k(w)\leq d_0(w)=\phi_{01}(w)<\phi_1(w)$ on set of positive probability. Hence we get $R_1(\phi_{01},\sigma_1,\sigma_2)<R_1(\phi_1,\sigma_1,\sigma_2)$ which implies $R_1(\delta_{\phi_{01}},\sigma_1,\sigma_2)<R_1(\delta_{\phi_1},\sigma_1,\sigma_2)$. This completes the proof of the result.
	
	\begin{corollary}
		The risk of the estimator %$\delta_{S1}=\min\left\{c_0\frac{S_1}{S_2},d_0(w)\frac{S_1+T}{S_2} \right\}=\min\left\{c_0,d_0(w)\frac{S_1+T}{S_1} \right\}\frac{S_1}{S_2}=\min\left\{c_0,d_0(W)\frac{1}{W} \right\}V$
		$\delta_{S1}=\min\left\{c_0,\frac{\alpha_1}{W} \right\}\frac{S_1}{S_2}$ dominates the BEE $\delta_{0}$ provided $P\left(\frac{\alpha_1}{c_0}<W\right)\neq 0$ holds true.
	\end{corollary}
	We now obtained the corresponding results for special loss functions in the following examples.
	\begin{example}
		\begin{itemize}
			\item [(i)]	For the quadratic loss function $L_1(t)$, the improved estimator is obtained as
			\begin{align*}
				\delta_{S11}=\min\left\{\frac{p_2-4}{p_1+2},\ \frac{p_2-4}{p_1+3}\frac{1}{W}\right\}\frac{S_1}{S_2}=\frac{p_2-4}{p_1+2}\min\left\{ 1,\ \frac{p_2+2}{p_1+3}\frac{1}{W}\right\}\frac{S_1}{S_2}.
			\end{align*}
			\item [(ii)] For the entropy loss function $L_2(t)$, the improved estimator is obtained as
			\begin{align*}
				\delta_{S12}=\min\left\{\frac{p_2-1}{p_1},\ \frac{p_2-2}{p_1+1}\frac{1}{W}\right\}\frac{S_1}{S_2}=\frac{p_2-1}{p_1}\min\left\{ 1,\ \frac{p_1}{p_1+1}\frac{1}{W}\right\}\frac{S_1}{S_2}.
			\end{align*}
			\item [(iii)] Under the symmetric loss function $L_3(t)$, the improved estimator is obtained as
			\begin{align*}
				\delta_{S13}&=\min\left\{\left(\frac{p_2(p_2-2)}{p_1(p_1-2)}\right)^{1/2},\ \left(\frac{p_2(p_2-2)}{(p_1+1)(p_1-1)}\right)^{1/2}\frac{1}{W}\right\}\frac{S_1}{S_2}\\
				&= \left(\frac{p_2(p_2-2)}{p_1(p_1-2)}\right)^{1/2}\min\left\{ 1,\ \left(\frac{p_1(p_1-2)}{(p_1+1)(p_1-1)}\right)^{1/2}\frac{1}{W}\right\}\frac{S_1}{S_2}.
			\end{align*}
		\end{itemize}
	\end{example}

	We obtained a Stein-type estimator $\delta_{S1}$ whose risk is nowhere larger than that of $\delta_{01}$ under a general loss function $L(t)$. In fact such estimator of $\lambda$	is a member of a more general class of estimator of the form (see \cite{das1992improved})
	\begin{align}\label{st2}
		\delta_{\varphi_2}=\varphi_2(U)\frac{S_1}{S_2}\ \mbox{where} \ U=\frac{\sqrt{n_1n_2}|\bar{Y}|}{\sqrt{S_1}}=\frac{\sqrt{T}}{\sqrt{S_1}}.
	\end{align}
	In the next section we will derive some improved estimator on this class.
	\section{Smooth improved estimator}\label{se4}

	In this section, we obtain a smooth improved estimator using the IERD method of \cite{kubokawa1994double} and \cite{kubokawa1994unified}. The proposed approach is carried out as follows. The joint density of $V$ and $U$ given $K$ is 
	\begin{align}
		f_k(v,u) \propto u^{2k} v^{\frac{p_1+2k-1}{2}} \left(1+\left(1+u^2\right)v\right)^{-\frac{p_1+p_2+2k+1}{2}},~~ v>0, u>0, k=0,1,2,\dots
	\end{align}
	We define 
	\begin{align*}
		F_k(v,y)=\int_{0}^{y} f_k(v,s) ds~ ~\mbox{and}~~ F_0(v,y)=\int_{0}^{y} f_0(v,s) ds.
	\end{align*}
	\begin{theorem}\label{thierd}
		Suppose the function $\varphi_2$ satisfies the following conditions.
		\begin{enumerate}
			\item [(i)] $\varphi_2(u)$ is non-decreasing function in $u$ and $\lim\limits_{u \rightarrow \infty}\varphi_2(u)=c_0$
			\item [(ii)] $\int_{0}^{\infty} L'(\varphi_2(u)v)vF_0(v,u)dv \geq 0$.
		\end{enumerate}
		Then the risk of the estimator $\delta_{\varphi_2}$ in (\ref{st2}) is smaller than the $\delta_{0}$ under the loss function $L(t)$.
	\end{theorem}
	\noindent\textbf{Proof.} Let RD be the risk difference of $\delta_{0}$ and $\delta_{\varphi_2}$, then we can express RD as
	$$\begin{aligned}
		RD & =E\left\{\int_1^{\infty} \frac{d}{d t}L\left(\varphi_{2}(tu)v\right) d t\right\}\\
		& =E\left\{\int_1^{\infty} L'\left(\varphi_{2}(tu)v\right) \varphi_2^{\prime}\left(tu\right) v ud t\right\}\\
		&=\int_{0}^{\infty}\int_{0}^{\infty}\int_1^{\infty} L'\left(\varphi_{2}(tu)v\right) \varphi_2^{\prime}\left(tu\right) vu f_{k}(v,u)dt du dv.\\
	\end{aligned}$$
	Making the transformation $z=tu$ and put $x=\frac{z}{t}$ we get
	%$$\begin{aligned}
		%	RD&\geq \int_{0}^{\infty}\int_{0}^{\infty}\int_1^{\infty} L'\left(\phi_{1}(h)v_1^{\frac{k}{2}}\right) \phi_1^{\prime}\left(h\right) v_1^{\frac{k}{2}} \frac{h}{t}f_{\eta}(\frac{h}{t},v_1)dt \frac{dh}{t} dv_1
		%\end{aligned}$$
		%Again put $q=\frac{h}{t}$,
		$$\begin{aligned}
			RD
			%	&\geq \int_{0}^{\infty}\int_{0}^{\infty}\int_0^{h} L'\left(\phi_{1}(h)v_1^{\frac{k}{2}}\right) \phi_1^{\prime}\left(h\right) v_1^{\frac{k}{2}} f_{\eta}(q,v_1)dq dh dv_1\\
			&=\int_{0}^{\infty}\varphi_2^{\prime}\left(z\right)\int_{0}^{\infty}\int_0^{z} L'\left(\varphi_{2}(z)v\right)  v f_{k}(v,x)dx dz dv.
		\end{aligned}$$
		It can be easily check that $\frac{f_{k}(v,x)}{f_{0}(v,x)}$ is non decreasing in $x$ $\forall~ x>0$, and hence we we have $\frac{F_{k}(v,x)}{F_{0}(v,x)}$ is non decreasing in $x$ for $x>0$. By using Lemma $2.1$ of \cite{kubokawa1994unified} and let $v^{*}=\varphi_2^{-1}$ we get 
		$$\begin{aligned}
			RD & \geq \int_{0}^{\infty}\varphi_2^{\prime}\left(z\right)\frac{F_{k}(v^{*},z)}{F_{0}(v^{*},z)}\int_{0}^{\infty} L'\left(\varphi_{2}(z)v\right) v F_{0}(v,z) dv dz.
		\end{aligned}$$
		Hence, the risk difference is non-negative if 
		$$\begin{aligned}
			\int_{0}^{\infty} L'\left(\varphi_{2}(z)v\right) v F_{0}(v,z) dv&\geq0.
		\end{aligned}$$
		This completes the proof.

		Now, we obtain improved estimators for $\lambda$ under three special loss functions by applying Theorem \ref{thierd} in the subsequent corollaries.
		\begin{corollary}
			Let us assume that the function $\varphi_2(u)$ satisfies the following conditions 
			\begin{enumerate}
				\item[(i)] $\varphi_2(u)$ is non-decreasing function in $u$ and $\lim\limits_{u\rightarrow\infty}\varphi_2(u)=\frac{p_2-4}{p_1+2}$
				\item[(ii)] $\varphi_2(u)\geq\varphi_{2}^{*1}(u)$
			\end{enumerate}
			where
			\begin{align*}
				\varphi_{2}^{*1}(u) = \frac{ \left(p_2-4\right) \int_{0}^{u} (1+s^2)^{-\frac{p_1+3}{2}}ds}{\left(p_1+3\right)\int_{0}^{u} (1+s^2)^{-\frac{p_1+5}{2}}ds}
			\end{align*} 
			The the risk of the estimator $\delta_{\varphi_2}$ given in (\ref{st2}) is nowhere greater than that of $\delta_{01}$ under the quadratic loss function $L_1(t)$.
		\end{corollary}
		
		\begin{corollary}
			Under the scale-invariant entropy loss function $L_2(t)$, the risk of the estimator $\delta_{\varphi_2}$ given in (\ref{st2}) is nowhere greater than that of $\delta_{02}$ provided the function $\varphi_2(u)$ satisfies  
			\begin{enumerate}
				\item [(i)] $\varphi_2(u)$ is non-decreasing function in $u$ and $\lim\limits_{u\rightarrow\infty}\varphi_2(u)=\frac{p_2-2}{p_1}$
				\item[(ii)] $\varphi_2(u)\geq\varphi_{2}^{*2}(u)$
			\end{enumerate}
			where
			\begin{align*}
				\varphi_{2}^{*2}(u) = \frac{ \left(p_2-2\right) \int_{0}^{u} (1+s^2)^{-\frac{p_1+1}{2}}ds}{\left(p_1+1\right) \int_{0}^{u} (1+s^2)^{-\frac{p_1+3}{2}}ds}.
			\end{align*} 
			
		\end{corollary}
		\begin{corollary}
			Suppose the following conditions hold true.  
			\begin{enumerate}
				\item [(i)] $\varphi_2(u)$ is non-decreasing function in $u$ and $\lim\limits_{u\rightarrow\infty}\varphi_2(u)=\sqrt{\frac{p_2(p_2-2)}{p_1(p_1-2)}}$
				\item[(ii)] $\varphi_2(u)\geq\varphi_{2}^{*3}(u)$
			\end{enumerate}
			where
			\begin{align*}
				\varphi_{2}^{*3}(u) 
				= \sqrt{\frac{p_2\left(p_2-2\right) \int_{0}^{u} (1+s^2)^{-\frac{p_1+1}{2}}ds}{\left(p_1+1\right)\left(p_1-1\right) \int_{0}^{u} (1+s^2)^{-\frac{p_1+3}{2}}ds}}.
			\end{align*}
			Then the risk of the estimator $\delta_{\varphi_2}$ given in (\ref{st2}) is nowhere greater than that of $\delta_{03}$ under the scale-invariant symmetric loss function $L_3(t)$. 
		\end{corollary}
		\begin{remark}
			In the above corollaries, we have obtained a call of improved estimators for $L_1(t),~L_2(t)$ and $L_3(t)$. The boundary estimators of this class are obtained as $\delta_{\varphi^{*1}_{2}}$, $\delta_{\varphi^{*2}_{2}}$ and $\delta_{\varphi^{*3}_{2}}$. These estimators are \cite{brewster1974improving} type estimators. 
		\end{remark}
		\section{A simulation study for improved estimator}\label{sim1}
		In this section, we conduct a Monte Carlo simulation study to compare the risk performance of the proposed estimators for small and moderate sample sizes. We were generated $60,000$ random samples from two normal populations $N(0,\sigma_1^2)$ and $N(0,\sigma_2^2)$ for various values of $\sigma_1$ and $\sigma_2$. As we can see from the above, each estimators depend on the scale parameter $\sigma_1$ and $\sigma_2$ only through $\sigma_r$. Hence, we conducted the analysis as a function of $\sigma_r$. To compare the performance of the estimators, we use the relative risk improvement (RRI) with respect to BEE. The RRI of an estimator $\delta$ relative to $\delta_0$ is defined by
		\begin{align*}
			\mbox{RRI}=\frac{\mbox{Risk}(\delta_0) - \mbox{Risk}(\delta)}{\mbox{Risk}(\delta_0)} \times 100.
		\end{align*}
		We have plotted the RRIs of various estimator of $\lambda$ with respect to different loss functions $L_1(t)$, $L_2(t)$ and $L_3(t)$ in Figures \ref{fig1}--\ref{fig3} respectively. Now we discussed the following key observation based on the figure \ref{fig1} for the loss function $L_1(t)$. 
		%   \begin{itemize}
			%   	\item [(i)] The improvement region of $\delta_{S1}$ is greater than $\delta_{\varphi_2^{*1}}$ when $\mu$ is near to origin. The improvement region of $\delta_{S1}$ is lesser than the $\delta_{\varphi_2^{*1}}$ when $\mu\geq 0.5$ (approximately).
			%   	
			%   	\item [(ii)] The RRI values of the estimator $\delta_{S1}$ is subsequently decreases when the parameter $\mu$ goes away from zero.
			%   	
			%   	\item [(iii)] The improvement region of $\delta_{\varphi_2^{*1}}$ is greater than $\delta_{S1}$ when $\mu$ is away from origin. The estimator $\delta_{\varphi_2^{*1}}$ have highest RRI value when $\mu$ takes moderate value such as $\mu=0.5$ (approximately).
			%   	
			%   	\item [(iv)] The RRI of $\delta_{S1}$ is increasing in $\sigma_r$ when the value of $\mu$ is away from zero i.e. $\mu\geq 0.4$ (approximately) where as when the value of $\mu$ takes near to zero the RRI of $\delta_{S1}$ is no longer a monotone function of $\sigma_r$.
			%   	
			%   	\item [(v)] The RRI of $\delta_{\varphi_2^{*1}}$ is decreasing in $\sigma_r$ when $\mu$ takes value near to zero i.e. $\mu\leq 0.4$ (approximately) where as when the value of $\mu \geq 0.4$ (approximately) the RRI of $\delta_{\varphi_2^{*1}}$ is no longer a monotone function of $\sigma_r$. 
			%   \end{itemize}
		%   
		%   It can be seen that the similar observation can be made for the loss function $L_2(t)$ and $L_3(t)$ see in Figures \ref{fig2} and \ref{fig3}.
		
		\begin{itemize}
			\item[(i)] The improvement region of $\delta_{S1}$ is larger than that of the estimator $\delta_{\varphi_2^{*1}}$ when $\mu$ is close to zero. However, when $\mu \geq 0.5$ (approximately), the improvement region of $\delta_{S1}$ is smaller than that of $\delta_{\varphi_2^{*1}}$.
			
			\item[(ii)] The RRI of the estimator $\delta_{S1}$ decreases as the parameter $\mu$ moves away from zero.
			
			\item[(iii)] The improvement region of $\delta_{\varphi_2^{*1}}$ is larger than that of $\delta_{S1}$ when $\mu$ is away from zero. Moreover, the estimator $\delta_{\varphi_2^{*1}}$ attains its highest RRI when $\mu$ takes a moderate value, approximately $\mu=0.5$.
			
			\item[(iv)] The RRI of $\delta_{S1}$ is increasing in $\sigma_r$ when $\mu$ is away from zero, i.e., $\mu \geq 0.4$ (approximately). In contrast, when $\mu$ is close to zero, the RRI of $\delta_{S1}$ is no longer monotone in $\sigma_r$.
			
			\item[(v)] The RRI of $\delta_{\varphi_2^{*1}}$ is decreasing in $\sigma_r$ when $\mu$ is close to zero, i.e., $\mu \leq 0.4$ (approximately). In contrast, when $\mu \geq 0.4$ (approximately), the RRI of $\delta_{\varphi_2^{*1}}$ is no longer monotone in $\sigma_r$.
		\end{itemize}
		Similar observations can be made for the loss functions $L_2(t)$ and $L_3(t)$; see Figures~\ref{fig2} and~\ref{fig3}, respectively.
	
		\begin{figure}[htbp]
			\centering
			\begin{subfigure}[b]{0.32\textwidth}
				\centering
				\includegraphics[width=\linewidth]{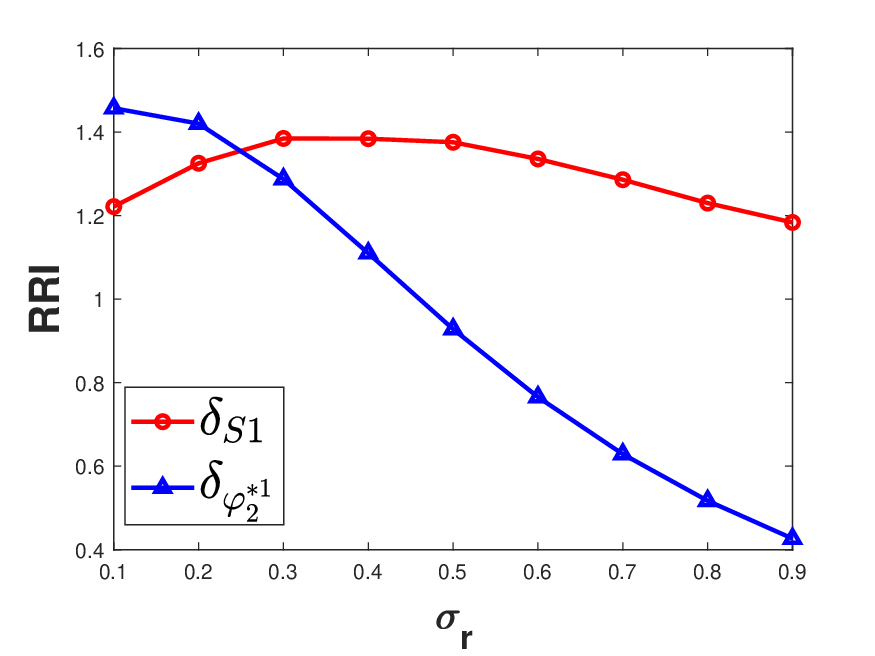}
				\caption{$(n_1,n_2)=(4,5)$, $\mu=0.3$}
			\end{subfigure}
			\hfill
			\begin{subfigure}{0.32\textwidth}
				\centering
				\includegraphics[width=\linewidth]{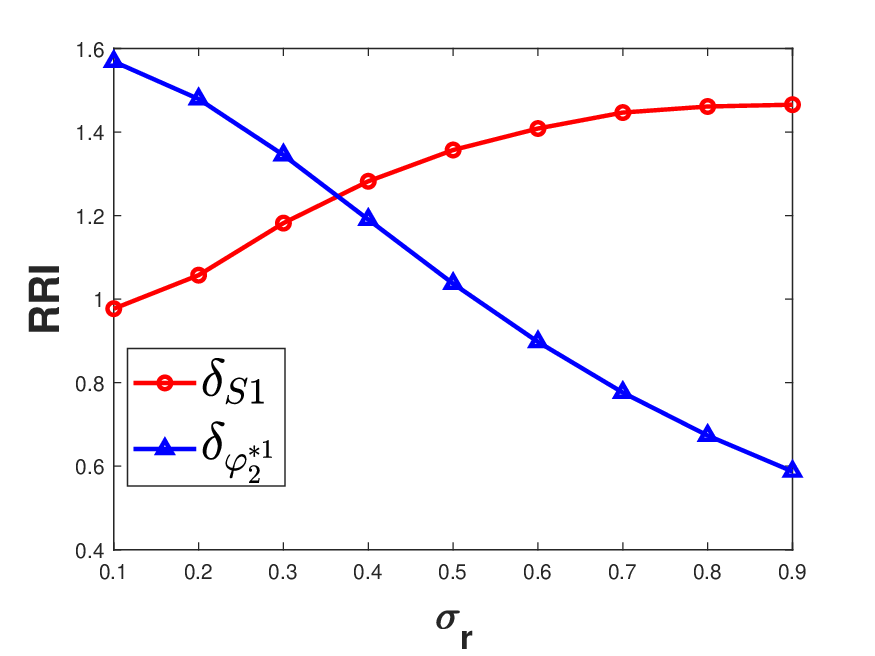}
				\caption{$(n_1,n_2)=(5,5)$, $\mu=0.3$}
			\end{subfigure}
			\hfill
			\begin{subfigure}{0.32\textwidth}
				\centering
				\includegraphics[width=\linewidth]{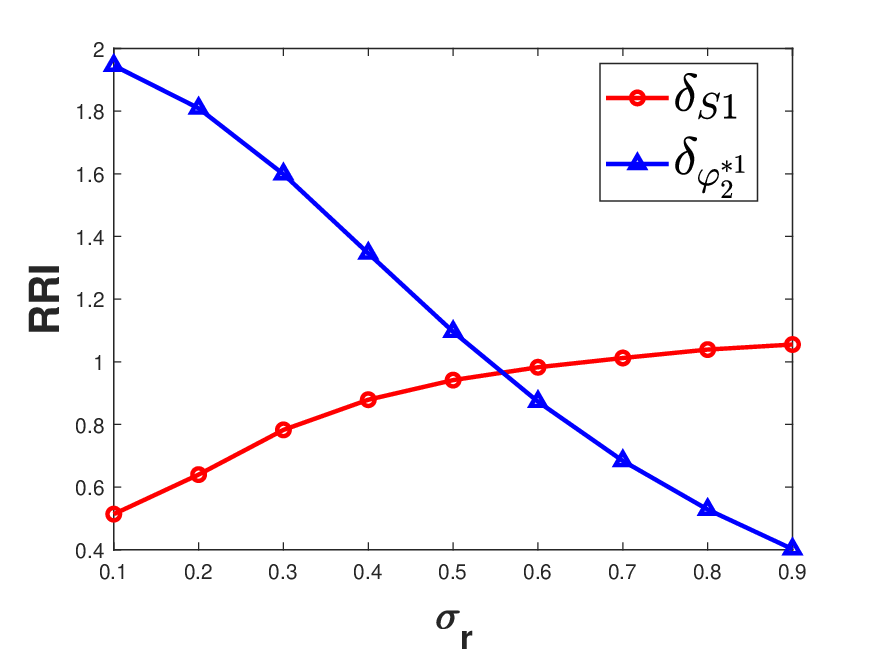}
				\caption{$(n_1,n_2)=(4,6)$, $\mu=0.4$}
			\end{subfigure}
			\vspace{0.3cm}
			\begin{subfigure}{0.32\textwidth}
				\centering
				\includegraphics[width=\linewidth]{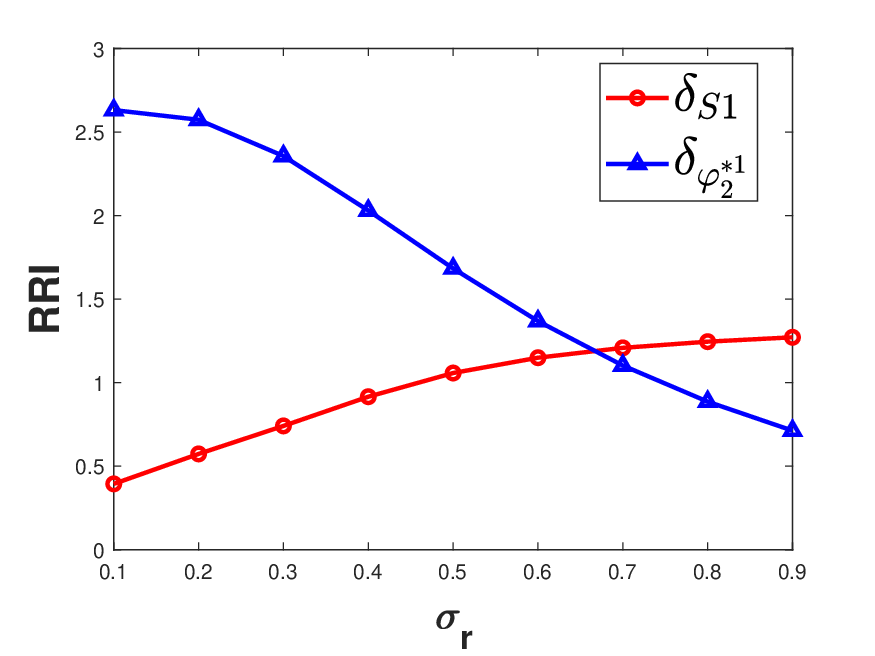}
				\caption{$(n_1,n_2)=(5,8)$, $\mu=0.4$}
			\end{subfigure}
			\hfill
			\begin{subfigure}{0.32\textwidth}
				\centering
				\includegraphics[width=\linewidth]{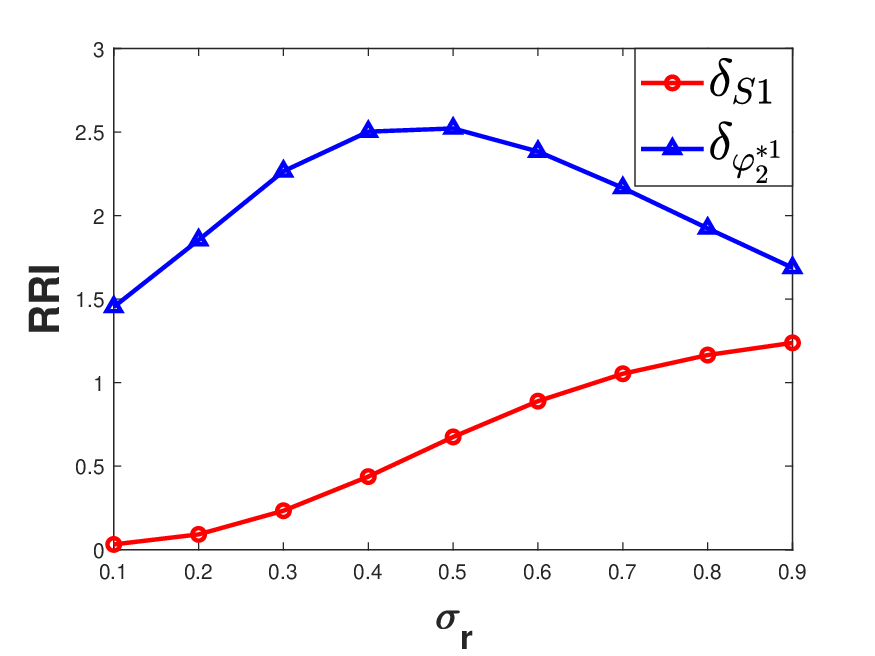}
				\caption{$(n_1,n_2)=(7,9)$, $\mu=0.5$}
			\end{subfigure}
			\hfill
			\begin{subfigure}{0.32\textwidth}
				\centering
				\includegraphics[width=\linewidth]{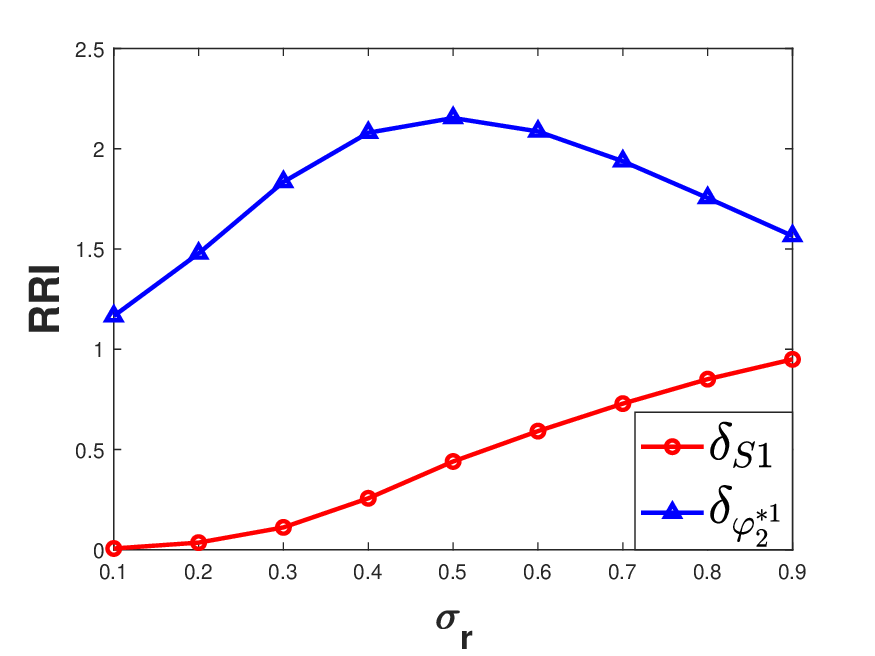}
				\caption{$(n_1,n_2)=(8,8)$, $\mu=0.5$}
			\end{subfigure}
			
			\vspace{0.3cm}
			
			\begin{subfigure}{0.32\textwidth}
				\centering
				\includegraphics[width=\linewidth]{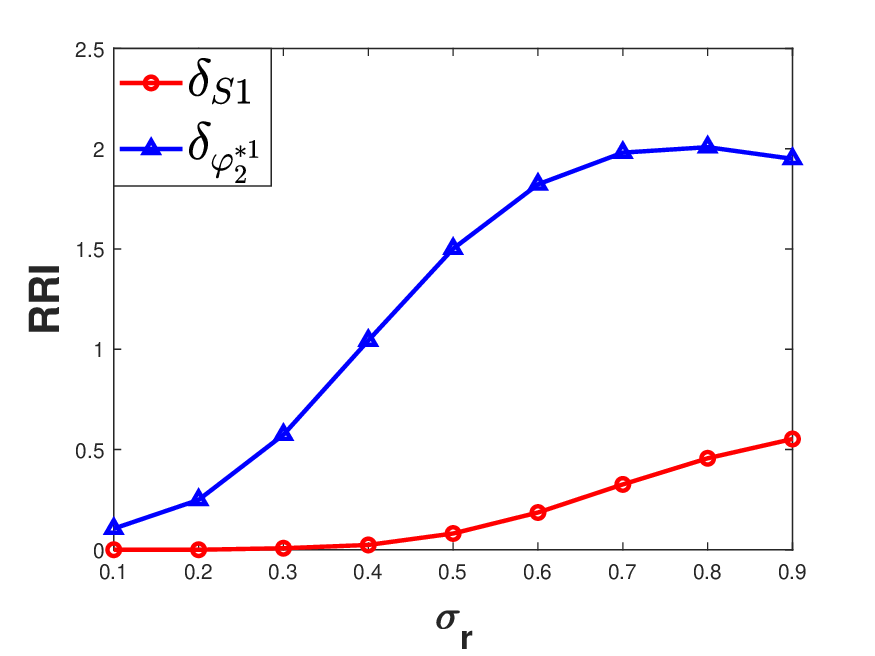}
				\caption{$(n_1,n_2)=(10,12)$, $\mu=0.6$}
			\end{subfigure}
			\hfill
			\begin{subfigure}{0.32\textwidth}
				\centering
				\includegraphics[width=\linewidth]{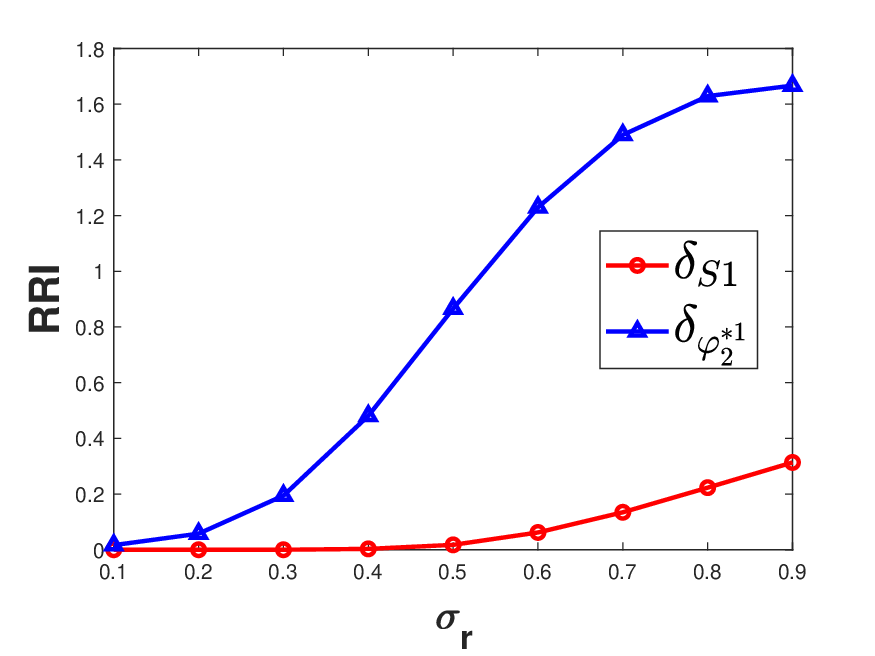}
				\caption{$(n_1,n_2)=(13,13)$, $\mu=0.6$}
			\end{subfigure}
			\hfill
			\begin{subfigure}{0.32\textwidth}
				\centering
				\includegraphics[width=\linewidth]{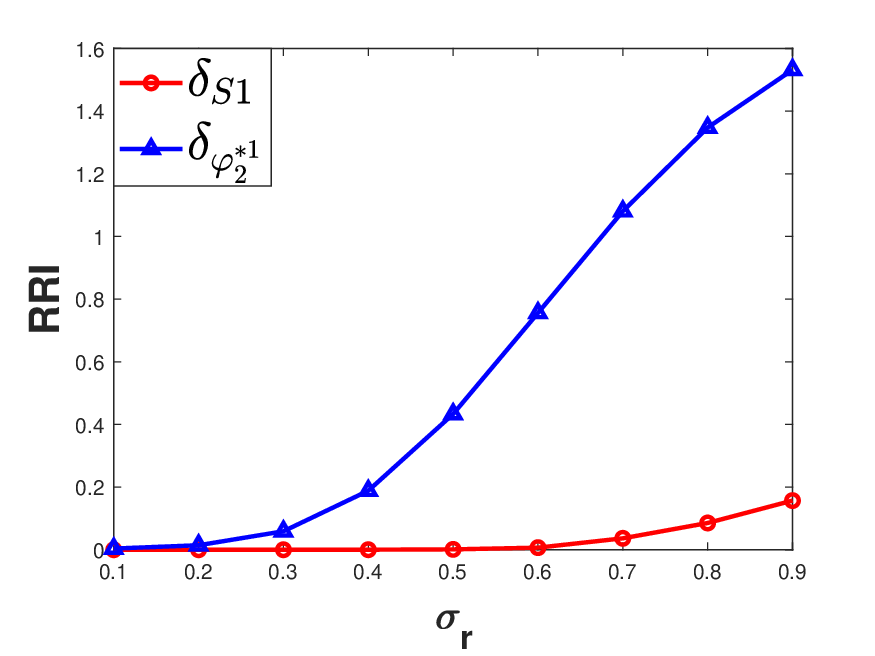}
				\caption{$(n_1,n_2)=(14,11)$, $\mu=0.7$}
			\end{subfigure}
			
			\caption{RRI of improved estimators for $\sigma_r$ under $L_1(t)$.}
			\label{fig1}
		\end{figure}

		\begin{figure}[t]
			\centering
			\begin{subfigure}{0.32\textwidth}
				\centering
				\includegraphics[width=\linewidth]{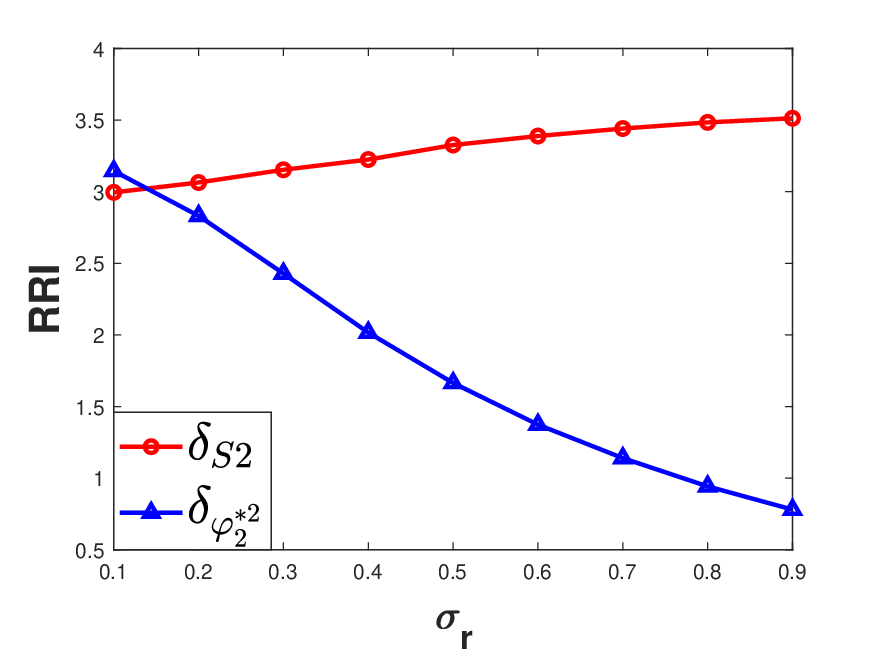}
				\caption{$(n_1,n_2)=(4,5)$, $\mu=0.3$}
			\end{subfigure}
			\hfill
			\begin{subfigure}{0.32\textwidth}
				\centering
				\includegraphics[width=\linewidth]{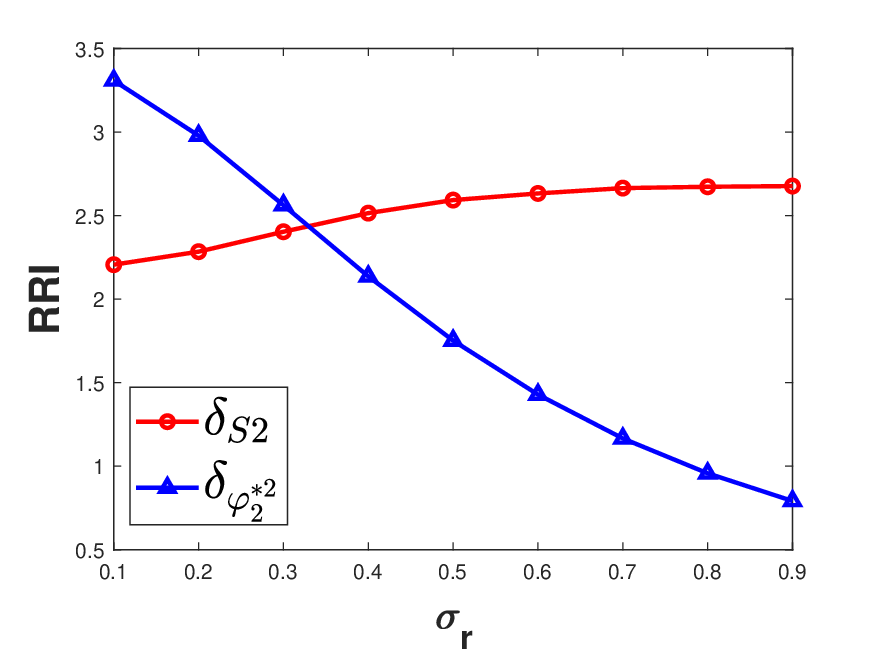}
				\caption{$(n_1,n_2)=(5,5)$, $\mu=0.3$}
			\end{subfigure}
			\hfill
			\begin{subfigure}{0.32\textwidth}
				\centering
				\includegraphics[width=\linewidth]{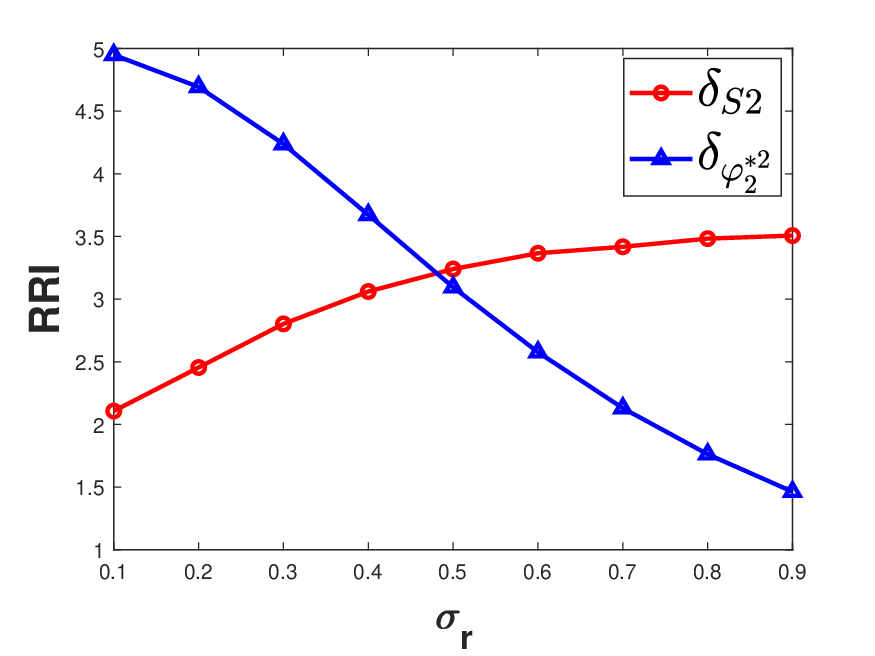}
				\caption{$(n_1,n_2)=(4,6)$, $\mu=0.4$}
			\end{subfigure}
			
			\vspace{0.3cm}
			
			\begin{subfigure}{0.32\textwidth}
				\centering
				\includegraphics[width=\linewidth]{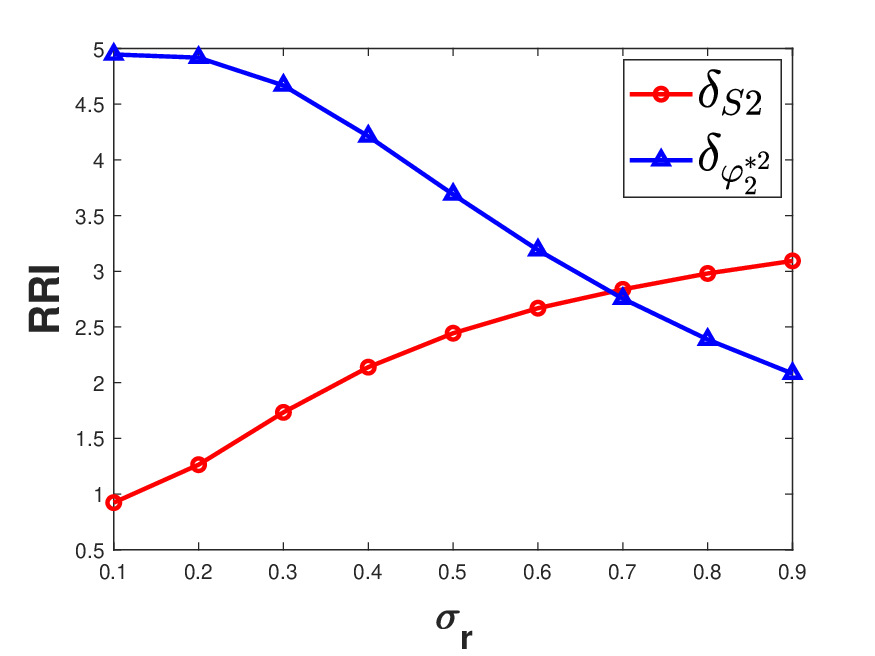}
				\caption{$(n_1,n_2)=(5,8)$, $\mu=0.4$}
			\end{subfigure}
			\hfill
			\begin{subfigure}{0.32\textwidth}
				\centering
				\includegraphics[width=\linewidth]{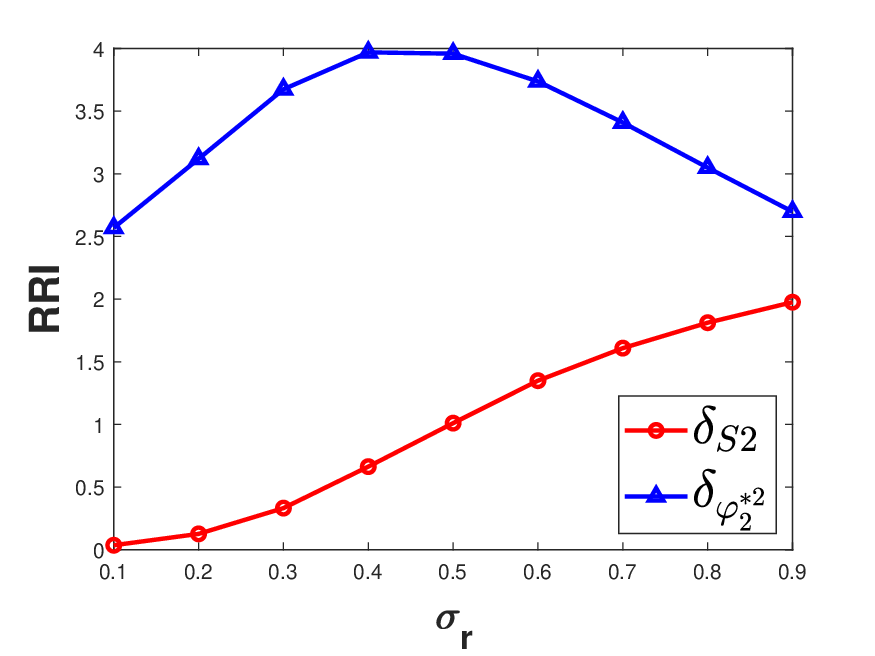}
				\caption{$(n_1,n_2)=(7,9)$, $\mu=0.5$}
			\end{subfigure}
			\hfill
			\begin{subfigure}{0.32\textwidth}
				\centering
				\includegraphics[width=\linewidth]{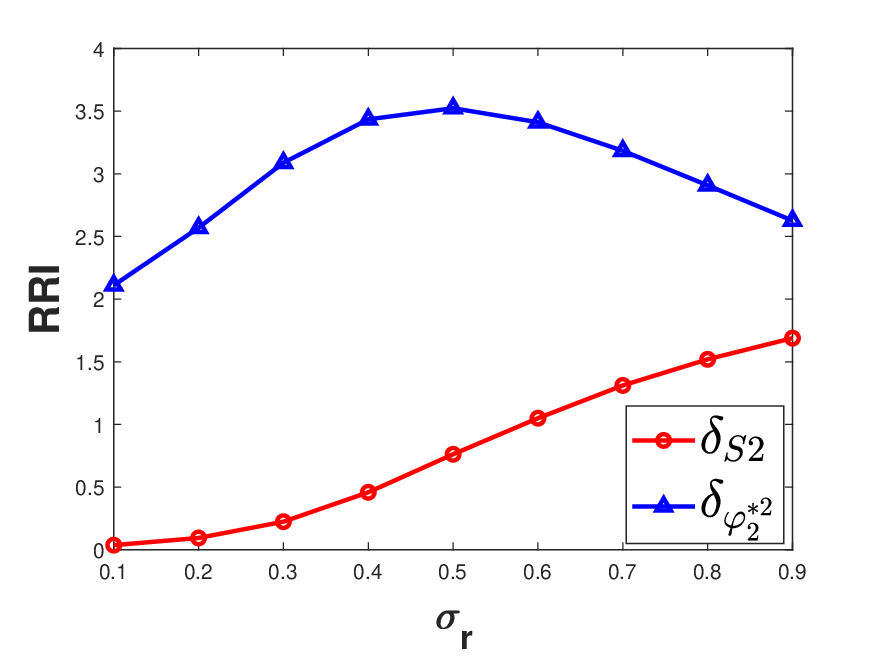}
				\caption{$(n_1,n_2)=(8,8)$, $\mu=0.5$}
			\end{subfigure}
			
			\vspace{0.3cm}
			
			\begin{subfigure}{0.32\textwidth}
				\centering
				\includegraphics[width=\linewidth]{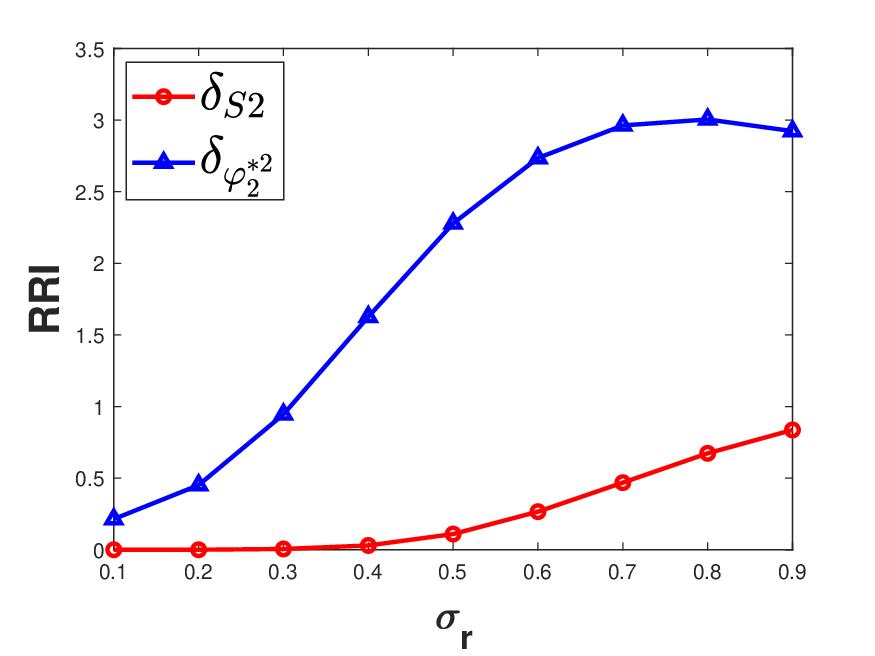}
				\caption{$(n_1,n_2)=(10,12)$, $\mu=0.6$}
			\end{subfigure}
			\hfill
			\begin{subfigure}{0.32\textwidth}
				\centering
				\includegraphics[width=\linewidth]{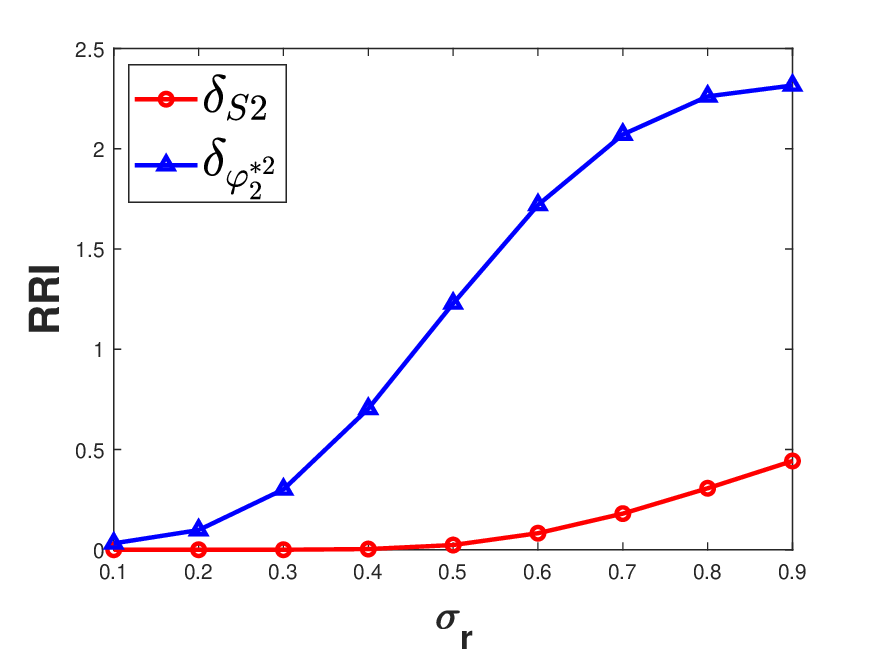}
				\caption{$(n_1,n_2)=(13,13)$, $\mu=0.6$}
			\end{subfigure}
			\hfill
			\begin{subfigure}{0.32\textwidth}
				\centering
				\includegraphics[width=\linewidth]{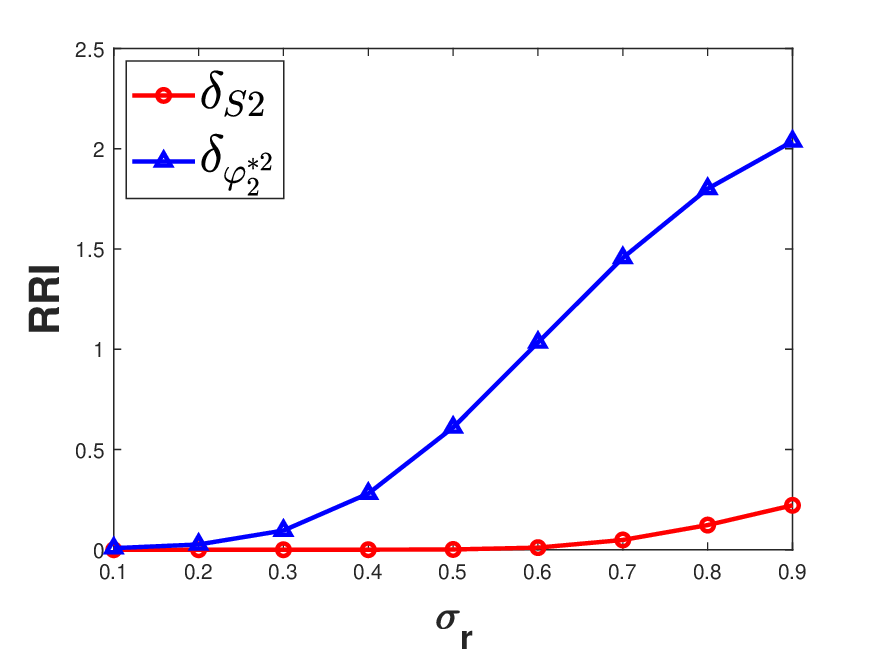}
				\caption{$(n_1,n_2)=(14,11)$, $\mu=0.7$}
			\end{subfigure}
			
			\caption{RRI of improved estimators for $\sigma_r$ under $L_2(t)$.}
			\label{fig2}
		\end{figure}
		
		\begin{figure}[t]
			\centering
			\begin{subfigure}{0.32\textwidth}
				\centering
				\includegraphics[width=\linewidth]{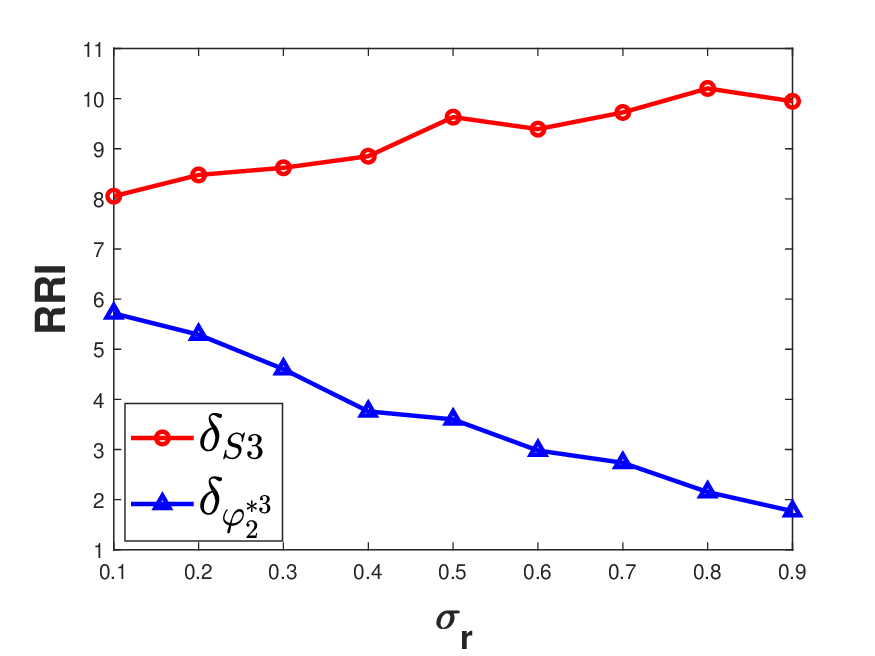}
				\caption{$(n_1,n_2)=(4,5)$, $\mu=0.3$}
			\end{subfigure}
			\hfill
			\begin{subfigure}{0.32\textwidth}
				\centering
				\includegraphics[width=\linewidth]{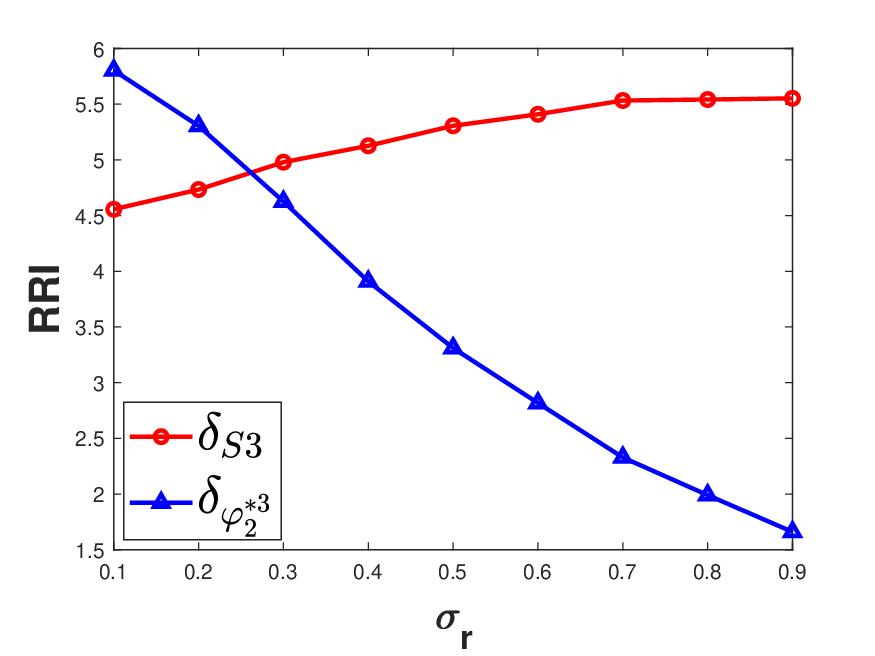}
				\caption{$(n_1,n_2)=(5,5)$, $\mu=0.3$}
			\end{subfigure}
			\hfill
			\begin{subfigure}{0.32\textwidth}
				\centering
				\includegraphics[width=\linewidth]{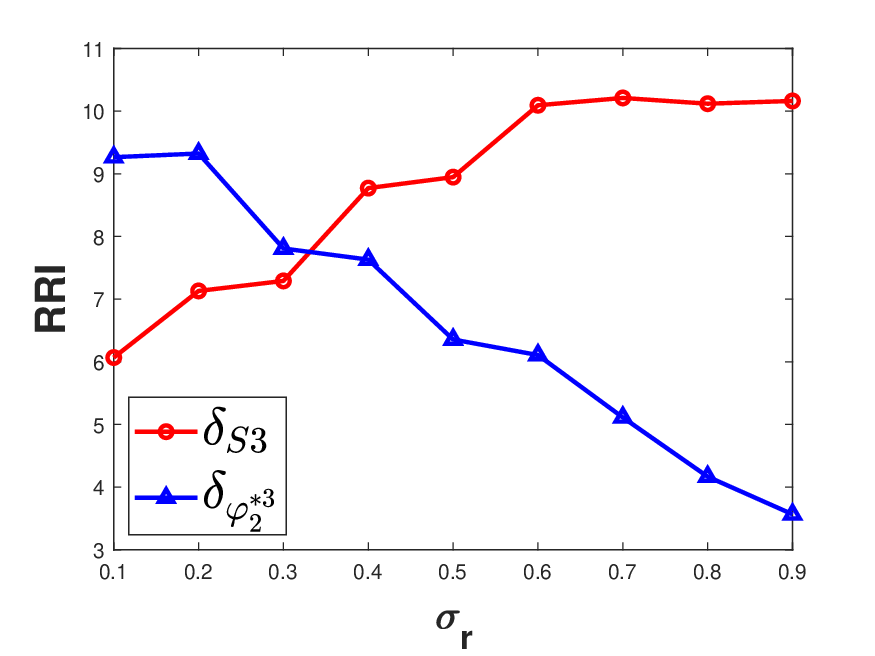}
				\caption{$(n_1,n_2)=(4,6)$, $\mu=0.4$}
			\end{subfigure}
			
			\vspace{0.3cm}
			
			\begin{subfigure}{0.32\textwidth}
				\centering
				\includegraphics[width=\linewidth]{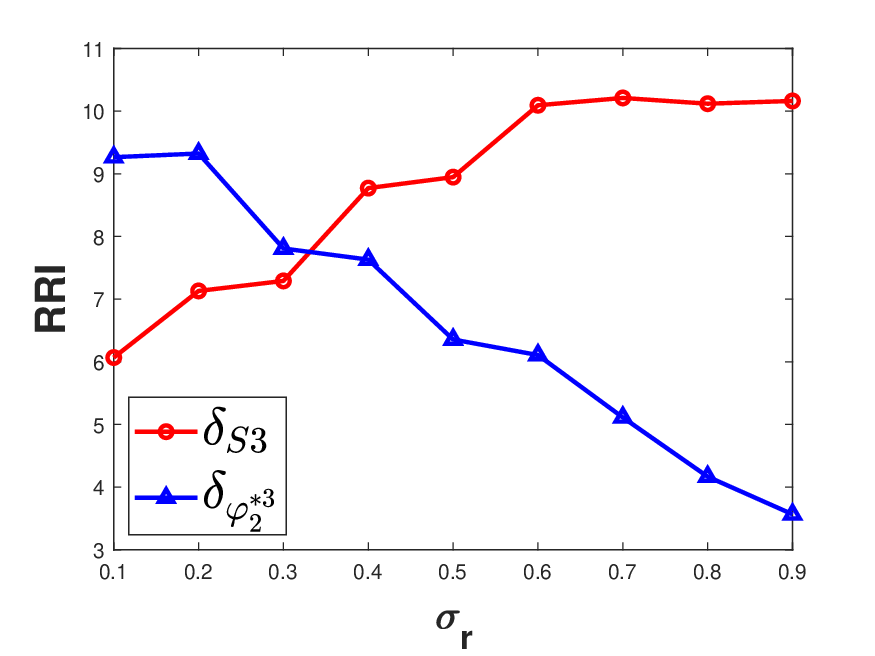}
				\caption{$(n_1,n_2)=(5,8)$, $\mu=0.4$}
			\end{subfigure}
			\hfill
			\begin{subfigure}{0.32\textwidth}
				\centering
				\includegraphics[width=\linewidth]{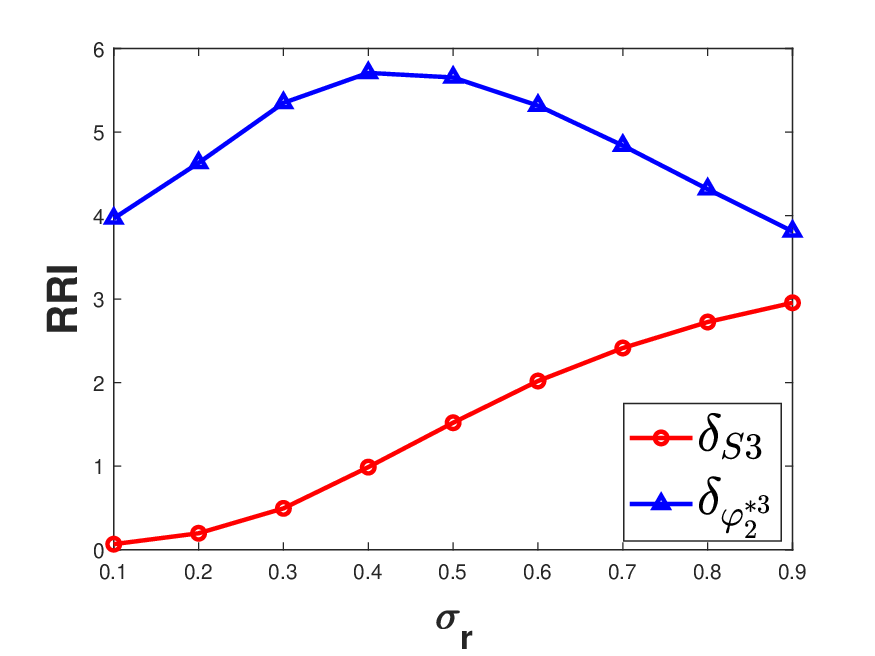}
				\caption{$(n_1,n_2)=(7,9)$, $\mu=0.5$}
			\end{subfigure}
			\hfill
			\begin{subfigure}{0.32\textwidth}
				\centering
				\includegraphics[width=\linewidth]{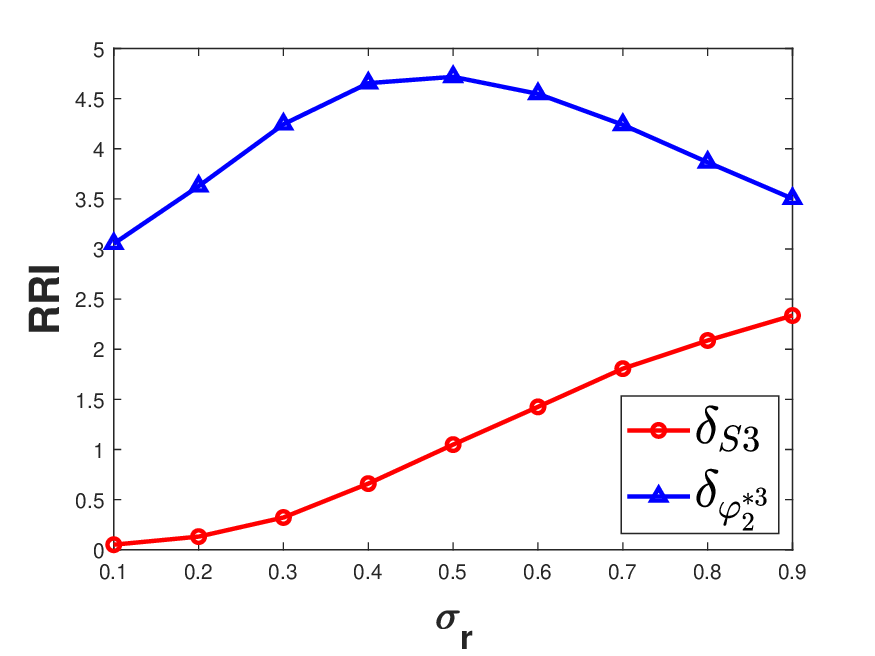}
				\caption{$(n_1,n_2)=(8,8)$, $\mu=0.5$}
			\end{subfigure}
			
			\vspace{0.3cm}
			\begin{subfigure}{0.32\textwidth}
				\centering
				\includegraphics[width=\linewidth]{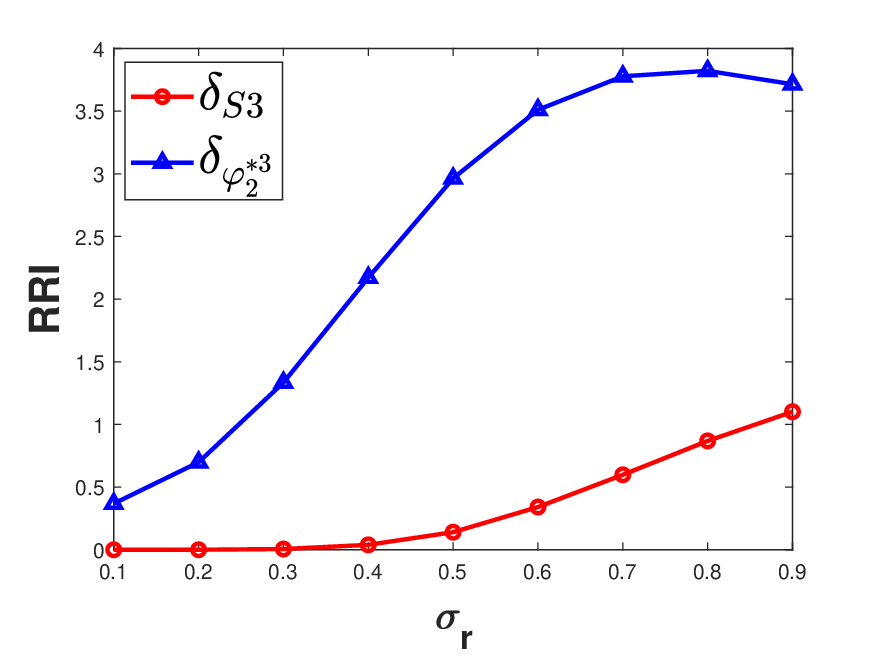}
				\caption{$(n_1,n_2)=(10,12)$, $\mu=0.6$}
			\end{subfigure}
			\hfill
			\begin{subfigure}{0.32\textwidth}
				\centering
				\includegraphics[width=\linewidth]{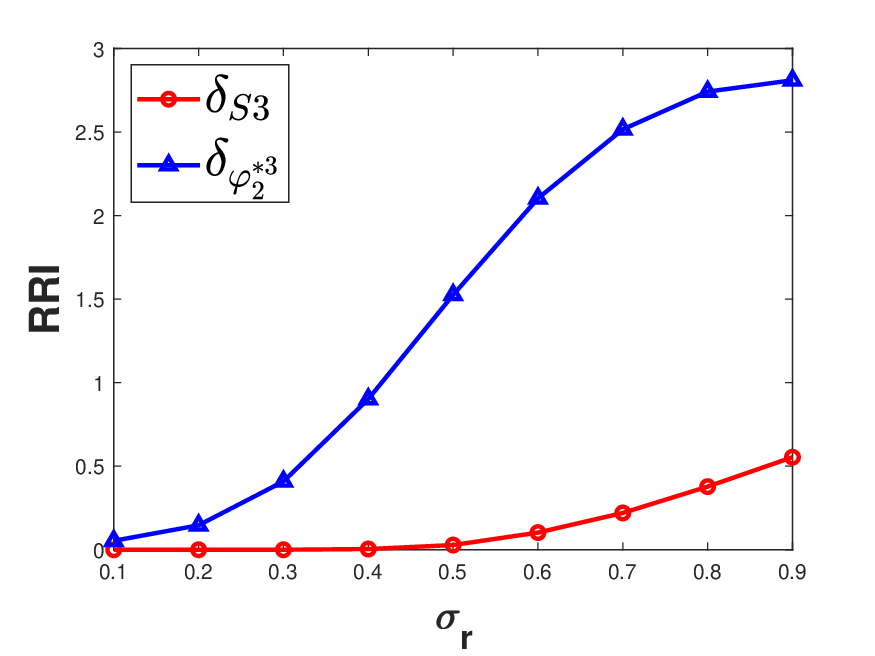}
				\caption{$(n_1,n_2)=(13,13)$, $\mu=0.6$}
			\end{subfigure}
			\hfill
			\begin{subfigure}{0.32\textwidth}
				\centering
				\includegraphics[width=\linewidth]{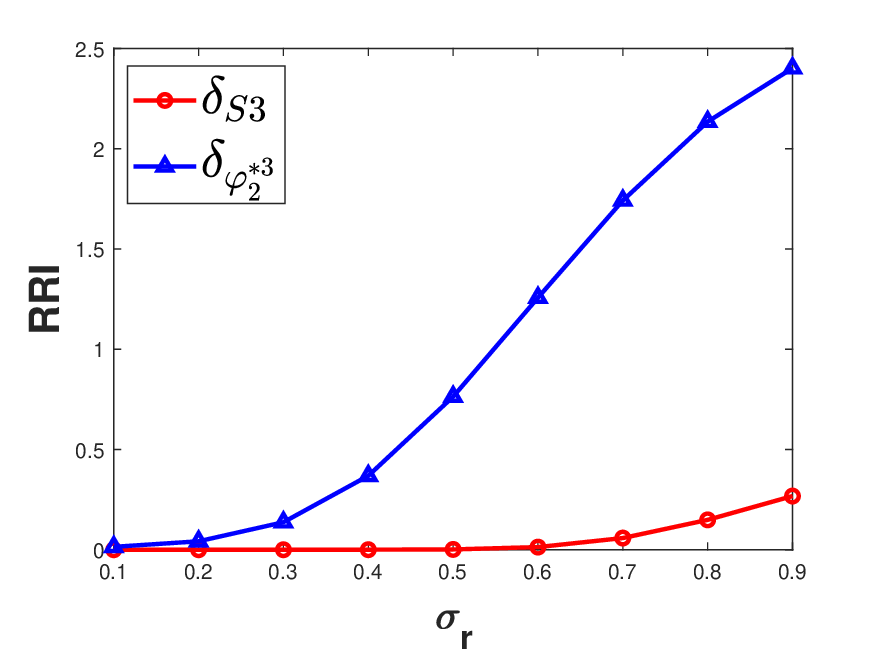}
				\caption{$(n_1,n_2)=(14,11)$, $\mu=0.7$}
			\end{subfigure}
			
			\caption{RRI of improved estimators for $\sigma_r$ under $L_3(t)$.}
			\label{fig3}
		\end{figure}
		\clearpage
		\section{Confidence intervals for $\sigma_r$}\label{se5}
		In this section, we present several methods for obtaining the confidence intervals for the ratio $\sigma_r$. In particular, we study asymptotic, bootstrap, generalized, and highest posterior density confidence intervals. Finally, simulation studies are conducted to examine the performance of these interval estimation methods. In the simulation study, the interval estimation methods are compared using coverage probability (CP), average length (AL), and probability coverage density (PCD) in order to identify the best-performing method across various sample sizes and parameter configurations.
		
		%#############################################################################
		\begin{align}
			y_{ij}=\mu+a_i+e_{ij},\ i=1,\dots,n_1;\ j=1,\dots,n_2,
		\end{align}
		where $a_i$ are independent $N(0,\sigma_1^2)$ and $e_{ij}$ are independent $N(0,\sigma_2^2)$ with $\sigma_1^2\geq 0$, $\sigma_2^2>0$ and $-\infty <\mu<\infty$. It follows that $y_i=(y_{i1}, y_{i2},\dots,y_{ij})$ are multivariate normal with mean $(\mu,\mu,\dots,\mu)$ and covariance matrix
		$\Sigma=\sigma^2\begin{pmatrix}
			1 & \rho & \dots & \rho\\
			\rho & 1 & \dots & \rho\\
			\vdots & \vdots & \vdots \\
			\rho & \rho & \dots & \rho
		\end{pmatrix}$
		where $\sigma^2=\sigma_1^2+\sigma_2^2$ and $\rho=\sigma_1^2/\sigma^2$. Using the factorization theorem and Basu's theorem a complete sufficient statistics for $\theta=(\mu,\sigma_1^2,\sigma_2^2)$ with independent coordinate is given by $S=(Y, S_1^2, S_2^2)$ where 	
		\begin{align*}
			&	Y=\sum_{i=1}^{n_1}\sum_{j=1}^{n_2}y_{ij}/(n_1n_2) \sim N(\mu, \tau^2/(n_1n_2))\\ 
			&	S_1^2=n_2\sum_{i=1}^{n_1}(y_i-y)^2 \sim \tau^2\chi^2_{n_1-1}\\
			&	S_2^2=\sum_{i=1}^{n_1}\sum_{j=1}^{n_2} (y_{ij}-y_i)^2 \sim \sigma_2^2\chi^2_{n_1(n_2-1)}
		\end{align*}
		where $\tau^2=\sigma_2^2+n_2\sigma_1^2$.
		%#############################################################################
		\subsection{Asymptotic confidence interval}\label{as1}
		In this subsection, we obtain an asymptotic confidence interval (ACI) for the parameter $\sigma_r$. To construct the asymptotic confidence interval (ACI), it is necessary to obtain the MLE of $\sigma_r$, which depends on the MLEs of the parameters $\mu$, $\sigma_1$ and $\sigma_2$. To derive the MLEs of these parameters, we consider the log-likelihood function based on the iid random sample $y_{ij}$
		\begin{align}
			l(\mu,\sigma_1,\sigma_2)=-\frac{n_1}{2}\ln(\tau^2)-\frac{n_1n_2(y-\mu)^2+s_1^2}{2\tau^2}-\frac{n_1(n_2-1)\ln \sigma_2^2}{2}-\frac{s_2^2}{2\sigma_2^2}
		\end{align}
		where $(y,s_1^2,s_2^2)$ are observed value of $(Y,S_1^2,S_2^2)$. The MLEs of $\mu$, $\tau^2$ $\sigma_1^2$, $\sigma_2^2$ are obtained as 
		\begin{align*}
			&\widehat{\mu}=Y, &\widehat{\tau}^2=\frac{S_1^2}{n_1},&& \widehat{\sigma}_2^2=\frac{S_2^2}{n_1(n_2-1)},&&\widehat{\sigma}_1^2=\frac{S_1^2}{n_1}-\frac{n_2S_2^2}{n_1(n_2-1)}.
		\end{align*}
		Since $\sigma_r=\frac{\sigma_1^2}{\sigma_2^2}$ then corresponding MLE of $\sigma_r$ can be written as $\widehat{\sigma}_r=\displaystyle\frac{\widehat{\tau}^2}{\widehat{\sigma}_2^2}-n_2=\frac{(n_2-1)S_1^2}{S_2^2}-n_2$. To derive the ACI for the ratio $\sigma_r$, we first obtain the Fisher information matrix. Let us denote the parameter by $\theta_1=\mu$, $\theta_2=\tau^2$, $\theta_3=\sigma_2^2$ and let $I(\theta_1,\theta_2,\theta_3)$ represents the Fisher information matrix. The $lm$-th entry is calculated as
		\begin{align*}
			I_{lm}=-E\left(\frac{\partial^2 l }{\partial\theta_l\partial \theta_m}\right);\ \  l,m=1,2,3
		\end{align*}
		After straightforward calculation, the Fisher information matrix can be expressed as follows
		\begin{align*}
			I(\mu,\tau^2, \sigma_2^2)=\begin{pmatrix}
				\frac{n_1n_2}{\tau} & 0&0\\
				0& \frac{n_1}{2\tau^4} &0\\
				0&0&\frac{n_1(n_2-1)}{2\sigma_2^4}
			\end{pmatrix}
		\end{align*}
		The inverse of $I(\mu,\tau^2,\sigma_2^2)$ is obtain as follows
		\begin{align*}
			I^{-1}(\mu,\tau^2, \sigma_2^2)=\begin{pmatrix}
				\frac{\tau}{n_1n_2}& 0&0\\
				0& \frac{2\tau^4}{n_1}&0\\
				0&0&\frac{2\sigma_2^4}{n_1(n_2-1)}
			\end{pmatrix}
		\end{align*}
		We now apply the well-known Delta method to derive the asymptotic confidence interval for the parameter $\sigma_r$. Let $g(\mu,\tau^2,\sigma_2^2)=\frac{\tau^2}{\sigma_2^2}-n_2$, and denote the gradient of $g$ by $\nabla g(\mu,\tau^2,\sigma_2^2)$. Then $\nabla g(\mu,\tau^2,\sigma_2^2)=\left(0,\frac{1}{\sigma_2^2}, -\frac{\tau^2}{\sigma_2^4}\right)^t$. Hence $\widehat{\sigma}_{rML}=g(\widehat{\mu}_{ML},\widehat{\tau}^2_{ML},\widehat{\sigma}_{2ML}^2)= \displaystyle \frac{\widehat{\tau}_{ML}^2}{\widehat{\sigma}_{2ML}^2}-n_2$ is asymptotically distributed as $N(\sigma_r, \nabla g^tI^{-1}\nabla g)$, where $ \nabla g^tI^{-1}\nabla g = \displaystyle\frac{2n_2(\sigma_r + n_2)^2}{n_1(n_2-1)}$. Using the above results, the $(1-\alpha)100\%$ asymptotic confidence interval for the ratio $\sigma_r$ is given by 
		\begin{align}
			\widehat{\sigma}_{rML}\pm z_{\frac{\alpha}{2}}\left(\widehat{\sigma}_{rML}+n_2\right) \sqrt{\frac{2n_2}{n_1(n_2-1)}}
		\end{align}
		where $z_{\frac{\alpha}{2}}$ is the upper $\alpha/2$ quantile of the standard normal distribution.
		
		In the following section, we present two additional approximate confidence intervals for the ratio $\sigma_r$.
		\subsection{Parametric bootstrap confidence intervals}
		The normal approximation performs reasonably well for large effective sample sizes. However, its accuracy may deteriorate when the sample size is small. In such cases, resampling methods, such as the bootstrap technique, can provide more accurate approximate confidence intervals. In this section, two bootstrap resampling methods are proposed for constructing approximate confidence intervals for the parameter $\sigma_r$.
		
		\subsubsection{Bootstrap-p confidence interval}\label{1}
		In this subsection, we construct confidence intervals using the percentile parametric bootstrap method, see \cite{efron1994introduction}. The bootstrap percentile method introduced by \cite{efron1982jackknife}. The following algorithm describes the procedure for constructing percentile bootstrap-p confidence intervals.
		\begin{itemize}
			\item [(i)] From the given observed sample $\{Y_{ij}\}$, compute the MLEs of the parameters $\mu$, $\sigma_1^2$, $\sigma_2^2$ as $\widehat{\mu}$, $\widehat{\sigma}_1^2$, $\widehat{\sigma}_2^2$ respectively.
			\item [(ii)] Generate a bootstrap sample $\{Y_{ij}^{*}=\widehat{\mu} + a_{i}^{*} + e_{ij}^{*}\}$. To generate bootstrap random sample, we use the MLEs. Specifically, the bootstrap random effects as $a_i^{*}\sim N(0,\widehat{\sigma}_{1ML}^2)$ and the bootstrap errors are generated as $e^{*}_{ij} \sim N(0,\widehat{\sigma}^2_{2ML})$. Then compute the bootstrap MLE of $\sigma_r$ using these bootstrap samples and obtain as $\nu^{*}=\widehat{\sigma}^{*}_{rML}=\displaystyle \frac{\widehat{\sigma}^{*2}_1}{\widehat{\sigma}_2^{*2}}$
			\item [(iii)] Repeat Step $(ii)$ independently $B$ times to obtain the bootstrap MLE of $\sigma_r$ as $\nu^{*}_1,\nu_2^{*},\dots,\nu_{B}^{*}$. Then arrange these estimates in the ascending order of magnitude.
			\item [(iv)] Let $F(x)=P(\nu^{*}\leq x)$ be the empirical bootstrap distribution function of $\widehat{\sigma}_{rML}$. Then the approximate $(1-\alpha)100\%$ bootstrap-p confidence interval of the ratio $\sigma_r$ is given by 
			\begin{align*}
				\left(
				\nu^{*}_{\left(\lceil B\alpha/2 \rceil\right)},
				\;
				\nu^{*}_{\left(\lceil B(1-\alpha/2) \rceil\right)}
				\right),
			\end{align*}
			where $F^{-1}(x)$ as the empirical quantile function of the bootstrap replicates $\nu_1^*, \dots, \nu_{B}^*$ (sorted ascending), so that the $\lceil Bx \rceil$-th order statistics estimates $F^{-1}(x)$.
		\end{itemize}
		In the following section, we present the algorithmic procedure for constructing the bootstrap-t confidence intervals for $\sigma_r$.
		
		\subsubsection{Bootstrap-t confidence interval}
		The bootstrap-t confidence interval was introduced by \cite{hall1988bootstrap}. The following steps describe the construction of an approximate confidence interval for $\sigma_r$ using the bootstrap-t method.
		\begin{itemize}
			\item [(i)] The first step is identical to Step $(i)$ of the bootstrap-p method in Sub-subsection \ref{1} where the MLEs of the parameter $\mu$, $\sigma_1^2$, $\sigma_2^2$ are computed. Using these MLEs obtain the MLE of the ratio $\sigma_r$ denoted by $\nu = \widehat{\sigma} _{rML}$. Also compute the asymptotic variance estimate from Subsection \ref{as1} as $	\widehat{\mbox{var}}(\nu)=\displaystyle\frac{2n_2(\nu + n_2)^2}{n_1(n_2-1)}.$ We denote $\widehat{\mbox{SE}}(\nu)=\displaystyle \sqrt{\frac{2n_2(\nu + n_2)^2}{n_1(n_2-1)}}.$
			\item [(ii)] Generate bootstrap sample $\{Y_{ij}^{*}=\widehat{\mu} + a_{i}^{*} + e_{ij}^{*}\}$ as Step $(ii)$ of Sub-subsection \ref{1}.
			\item [(iii)] For each bootstrap replicate $b=1,\dots B$, compute $\nu_b^*=\widehat{\sigma}_{rML}^{*(b)}$ then calculate the bootstrap-t statistic as $T_{b}^{*}=\displaystyle \frac{\nu_b^*-\nu}{\widehat{\mbox{SE}}(\nu_b^{*})}$.
			\item [(iv)] Repeat Steps $(ii)$ and $(iii)$ independently $B$ times and arrange the resulting bootstrap statistics as $T_1^*\leq T_2^* \leq \dots \leq T_B^*$.
			\item [(v)] Let $G(x)=P(T^*\leq x)$ be the empirical distribution $T_1^*, \dots, T^*_B$. Then the approximate $(1-\alpha)100\%$ bootstrap-t confidence interval for $\sigma_r$ is 
			\begin{align*}
				\left( \nu - T^*_{\lceil (1-\alpha / 2)B \rceil}\widehat{SE}(\nu),\ \nu - T^*_{ \lceil(\alpha / 2)B\rceil}\widehat{SE}(\nu) \right)
			\end{align*}
		\end{itemize}
		In the following subsection, we introduce the generalized variable approach and use it to derive additional confidence intervals.
		
		\subsection{Generalized confidence interval} %for the ratio $\sigma_r$}
	In this subsection, we discuss the generalized variable approach proposed by \cite{tsui1989generalized} and \cite{weerahandi1993generalized} for constructing confidence intervals. This method was successfully applied by \cite{krishnamoorthy2003inferences} for hypothesis testing and for constructing confidence intervals for the common mean $\mu$ of several normal populations. Related extensions of this approach have also been considered in studies such as \cite{jena2023point} and \cite{khatun2022hypothesis}. The following concepts play an important role in deriving the subsequent results.
	
	\begin{defn}\label{d1}
		Let $X$ be a random variable whose distribution depends on $(\theta,\varepsilon )$, where $\theta$ is the parameter of interest and $\varepsilon$ is a nuisance parameter. A statistic $T=T(X;x,\theta,\varepsilon)$ is called a generalized pivotal variable for constructing a generalized confidence interval for $\theta$ if it satisfies the following conditions:
		\begin{itemize}
			\item [(G1)] The distribution of $T(X;x,\theta,\varepsilon)$, for fixed $X=x$, is free from all unknown parameters.
			\item [(G2)] The observed value of $T(X;x,\theta,\varepsilon)$ at $X=x$ is equal to $\theta$, that is, $T(X;x,\theta,\varepsilon)=\theta$, where $\theta$ is the parameter of interest.
		\end{itemize}
	\end{defn}
	Let $(y,s_1^2,s_2^2)$ be an observed value of $(Y,S_1^2,S_2^2)$. Define $U_1=\frac{S_1^2}{\tau^2}\sim\chi^2_{n_1-1}$, $U_2=\frac{S_2^2}{\sigma_2^2}\sim\chi^2_{n_1(n_2-1)}$ where $\tau^2 =\sigma_2^2(1+n_2\sigma_r)$. The generalized pivot variable for $\sigma_r$ based on the sample $(y,s_1^2,s_2^2)$ is $T=\frac{1}{n_2}\left(\frac{s_1^2U_2}{s_2^2U_1}-1\right).$ It follows immediately from Definition \ref{d1} that $T$ satisfies both conditions of a generalized pivotal variable: its distribution is free from the unknown parameters, and its observed value is equal to the parameter of interest $\sigma_r$. %It can be verified from Definition \ref{d1} that the distribution of $T$ is free from the nuisance parameters and that $T(s,\theta)=\sigma_r$. 
	Therefore $100(1-\alpha)\%$ generalized confidence interval for $\sigma_r$ is given by $$\left(T_{ \left(\frac{\alpha}{2}\right) },\ T_{\left(1 - \frac{\alpha}{2}\right)}\right),$$ where $T_{\frac{\alpha}{2}}$ and $T_{1-\frac{\alpha}{2}}$ denote the corresponding quantiles of $T$.
	
	\subsection{HPD credible interval using MCMC method}
	We emphasize that Bayesian inference plays an important role when prior information about the behavior of the parameters is available in advance. In some situations, these intervals or point estimators may perform better than the classical methods based on MLEs. \cite{box2011bayesian} showed that, among all Bayesian credible intervals with credibility level $(1-\alpha)$, the HPD interval possesses the shortest average length (AL). In this subsection, we derive the HPD credible interval for the ratio $\sigma_r$ using Jeffrey's non-informative prior and an MCMC procedure for the unknown parameters. The Jeffrey's prior is given by
	\begin{align*}
		\pi(\theta)&=\sqrt{|I(\theta)|}, \ \ \theta=(\mu,\tau^2,\sigma_2^2)\\
		&\propto \frac{1}{\big(\tau^2\big)^{\frac{3}{2}} \sigma_2^2} \ \ \mbox{with support}\ \tau^2>0, \sigma_2^2>0, \mu\in\mathbb{R}
	\end{align*}
	The joint posterior density of $\theta$, for the observed sample values, is proportional to the product of the likelihood function and the Jeffrey's prior. From the joint posterior density of $\theta$, the corresponding conditional densities of the parameters are obtained as follows:
	\begin{itemize}
		\item [(i)] $\pi \big(\mu,\tau^2,\sigma_2^2|\mbox{Data}\big) \propto \big(\tau^2\big)^{-\frac{n_1+3}{2}} \big(\sigma_2^2\big)^{-\left(\frac{n_1(n_2-1)}{2}+1\right)} e^{\left(-\frac{n_1n_2(x-\mu)^2+s_1^2}{2\tau^2} - \frac{s_2^2}{2\sigma_2^2}\right)}$
		\item [(ii)] $\mu|\tau^2,\sigma_2^2,\mbox{Data} \sim N(x,\frac{\tau^2}{n_1n_2})$
		\item [(iii)] $\tau^2|\mu,\sigma_2^2,\mbox{Data}\sim IG\left(\frac{n_1+1}{2},\ \frac{n_1n_2(x-\mu)^2+s_1^2}{2}\right)$
		\item [(iv)] $\sigma_2^2|\mu,\tau^2,\mbox{Data} \sim IG\left(\frac{n_1(n_2-1)}{2},\ \frac{s_2^2}{2}\right)$
	\end{itemize}
	To construct an confidence interval for $\sigma_r$, we use the Gibbs sampling procedure to generate samples $\mu$. To construct a confidence interval for $\sigma_r$, we use the Gibbs sampling procedure \citep{gelfand1990sampling}, a special case of the MCMC framework. Since all three full conditional distributions are of standard form Normal for $\mu$ and Inverse-Gamma for $\tau^2$ and $\sigma_2^2$ each parameter is sampled directly and exactly from its full conditional at every iteration, requiring no Metropolis Hastings accept/reject step. The Gibbs sampling algorithm is outlined below.
	%The parameter $\tau^2$ and $\sigma_2^2$ will be generated using the well known Markov Chain Monte Carlo (MCMC) framework along with the Random Walk Metropolis-Hasting (RWMH) algorithm.Detailed descriptions of the Random Walk Metropolis–Hastings algorithm can be found in the works of \cite{hastings1970monte} and \cite{chib1995understanding}. The combined Gibbs sampling and Random Walk Metropolis–Hastings algorithm is outlined below.
	
	\textbf{Gibbs Sampling Algorithm:}
	Let the current state of the Markov chain at the $(m-1)$-th iteration be $\underline{\delta}^{(m-1)}=\left(\mu^{m-1}, \tau^{2(m-1)}, \sigma_2^{2(m-1)}\right)$. For $m=1,2,\dots,M$, generate $\mu^{(m)}\propto N\left(y,\frac{\tau^{2(m-1)}}{n_1n_2}\right)$; $\tau^{2(m)} \propto IG\left(\frac{n_1+1}{2},\ \frac{n_1n_2(y-\mu^{(m)})^2+s_1^2}{2}\right)$; $\sigma_2^{2(m)} \propto IG\left(\frac{n_1(n_2-1)}{2},\ \frac{s_2^2}{2}\right)$. At iteration $m$, generate $\theta_m=\frac{\tau^{2(m)}}{\sigma_2^{2(m)}}-n_2$. This produces posterior samples $\theta_1,\theta_2,\dots,\theta_P$, where $P=M-M_0$. Discard the first $M_0$ burn in draws and retain $P=M-M_0$ posterior samples. Sort these as $\theta_{(1)}\leq \theta_{(2)}\leq \dots \leq \theta_{(M)}$. Find $b^*$ that minimizes the interval width 
	\begin{align*}
		b^*=\arg \min \left(\theta_{(b+\lfloor(1-\alpha)P\rfloor)}-\theta_{b}\right), ~~ b = 1,2, \dots, P-\lfloor(1-\alpha)P\rfloor
	\end{align*}
	Then the $(1-\alpha)100\%$ HPD credible interval for $\sigma_r=\frac{\sigma_1^2}{\sigma_2^2}$ is $\left(\theta_{b^*}, \theta_{(b^*+ \lfloor (1-\alpha)P\rfloor)} \right)$.

	\section{A simulation study for confidence interval} \label{se6}
	In this section we will compare the performances of the derived confidence interval for $\sigma_r$. We use extensive Monte Carlo simulation methods to compare the coverage probability (CP) and average length (AL). We were generated 30,000 random sample from two normal populations $N(0,\sigma_1)$ and $N(0,\sigma_2)$. For computing the bootstrap confidence intervals, we were used 3000 iteration for inner loop corresponding to the replication. For the generalized confidence intervals, the inner loop consists of 10,000 replications. Similarly, for obtaining the HPD credible intervals using the MCMC algorithm, the inner loop is performed for 10,000 replications. furthermore we were computed all the confidence intervals at the significance level $(1-\alpha)=0.95$. When evaluating the performance of the confidence intervals based on their CPs and ALs, we find that some intervals have CPs close to the nominal level of 0.95 but relatively large ALs. On the other hand, some intervals have smaller ALs but their CPs are below 0.95. To deal with this issue and compare the intervals more effectively, we use the following methods. First, we choose a lower acceptable value for the nominal CP level and identify the confidence intervals that achieve this requirement. Among these intervals, we select the one with the smallest AL. The second option is to consider the confidence intervals that achieve the minimum acceptable nominal CP level and then evaluate them using a unified measure called the probability coverage density (PCD), which is defined as the ratio of CP to AL. This unified performance measure was proposed by \cite{unhapipat2016small}. It considers both criteria, CP and AL, simultaneously when evaluating the performance of the confidence intervals. This criterion gives us those intervals which have relatively higher CPs and smaller ALs. Hence, the interval having the higher PCD value implies the best performance. We have presented the results of the simulation study in Figures \ref{p1}--\ref{p5} for the respective confidence intervals. We have reported the simulation study for fixed values of $n_2=8$, $\mu=0$, $\sigma_1=1$ and  $\sigma_2=0.5$.
	
	The following conclusions can be drawn from the simulation study and Figures \ref{p1}--\ref{p5} regarding the performance of the confidence intervals in terms of their ALs, CPs, and PCDs.
	
	\begin{itemize}
		\item [(i)] The ALs decrease as $n_1$ increases for all the confidence intervals. However, the CPs decrease with increasing $n_1$ for the bootstrap-p, HPD credible, and generalized confidence intervals. For the asymptotic and bootstrap-t confidence intervals, the ALs do not show a monotonic pattern with respect to $n_1$.
		
		\item [(ii)] One can observed that the ALs of the confidence intervals bootstrap-p and bootstrap-t are coincide.
		
		\item [(iii)] If we fixed the confidence level at $95\%$, the asymptotic and HPD credible confidence intervals are achieve this.
		
		\item [(iv)] The asymptotic confidence interval provides the highest CP and the smallest AL when the sample size $n_1$ is moderate. When the sample size is small, the HPD credible interval provides the highest CP and the smallest AL.
		
		\item [(v)] The PCD values increase as $n_1$ increases for the asymptotic, generalized, and HPD credible intervals. However, the PCD values decrease for the bootstrap-p interval, while no clear pattern is observed for the bootstrap-t interval.
		
		\item [(vi)] In terms of PCD values, the bootstrap-t confidence interval gives the highest PCD value when $n_1 = 15$ (approximately). The asymptotic and bootstrap-p confidence intervals perform better when the sample size $n_1$ is small.
	\end{itemize}

	\begin{figure}[!t]
		\centering
		\includegraphics[width=0.9\textwidth]{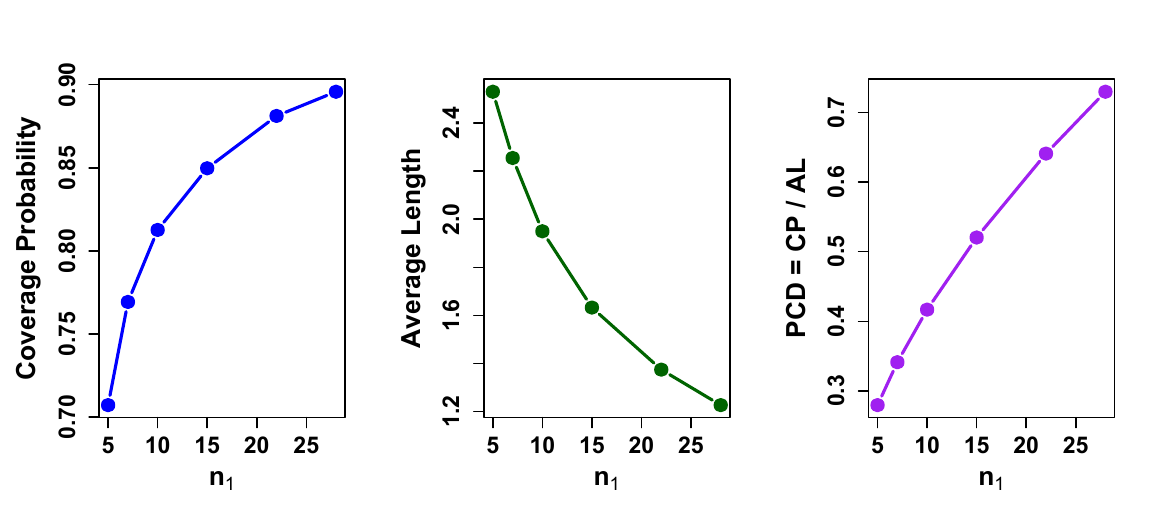}
		\caption{Asymptotic confidence interval for $n_2=8$, $\mu=0$, $\sigma_1=1, \sigma_2=1$}
		\label{p1}
		\vfill
		\includegraphics[width=0.9\textwidth]{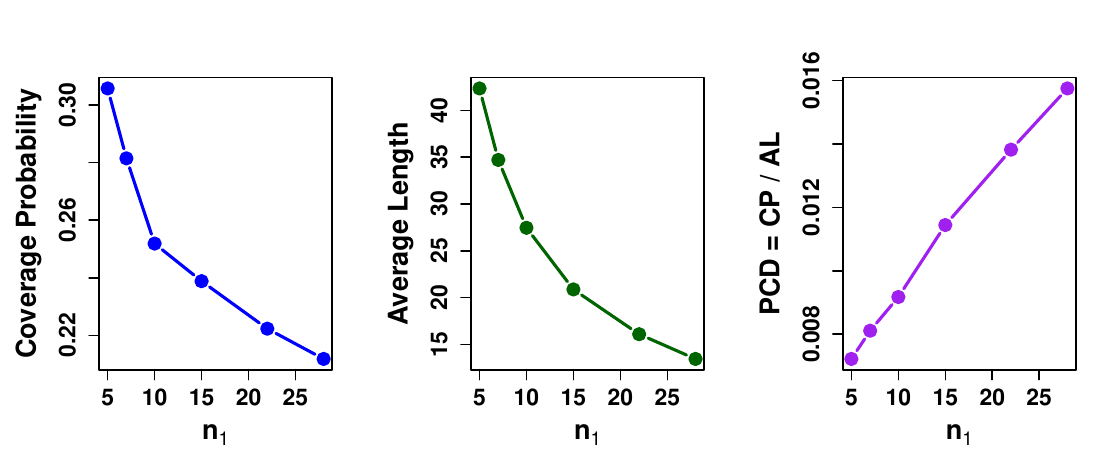}
		\caption{Bootstrap-p for $n_2=8$, $\mu=0$, $\sigma_1=1, \sigma_2=1$}
		\label{p2}
		\vfill
		\includegraphics[width=0.9\textwidth]{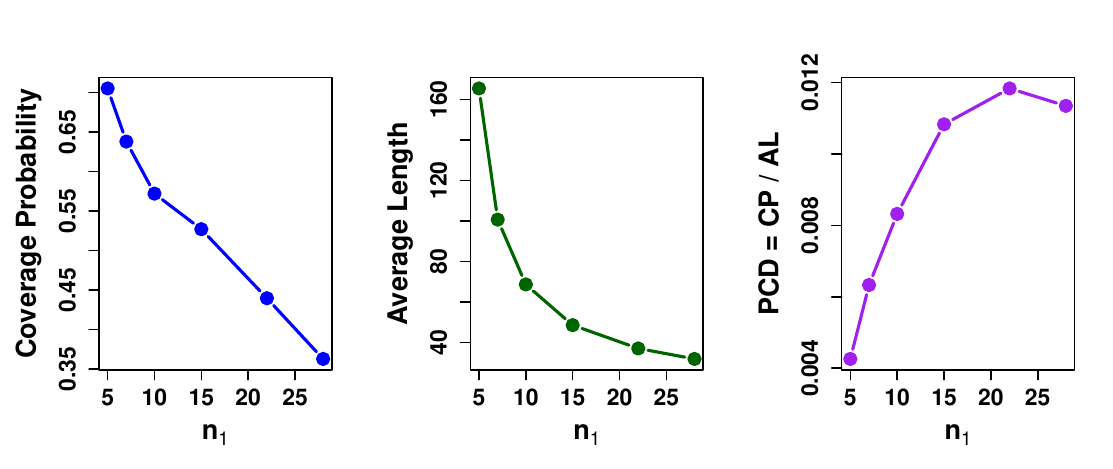}
		\caption{Bootstrap-t for $n_2=8$, $\mu=0$, $\sigma_1=1, \sigma_2=1$}
		\label{p3}
	\end{figure}
	\begin{figure}[!t]
		\centering
		\includegraphics[width=0.9\textwidth]{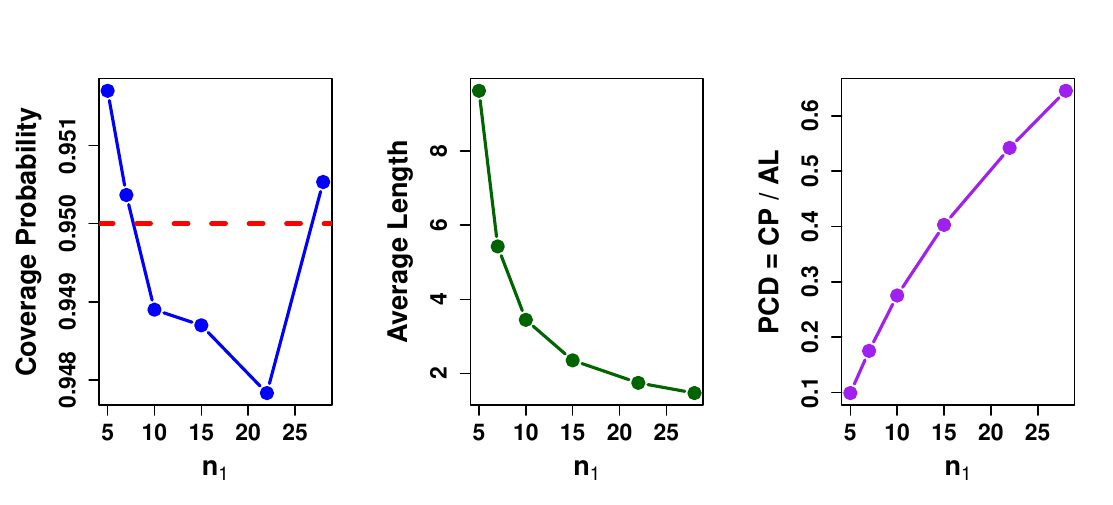}
		\caption{Generalized CI for $n_2=8$, $\mu=0$, $\sigma_1=1, \sigma_2=1$}
		\label{p4}
		\vfill
		\includegraphics[width=0.9\textwidth]{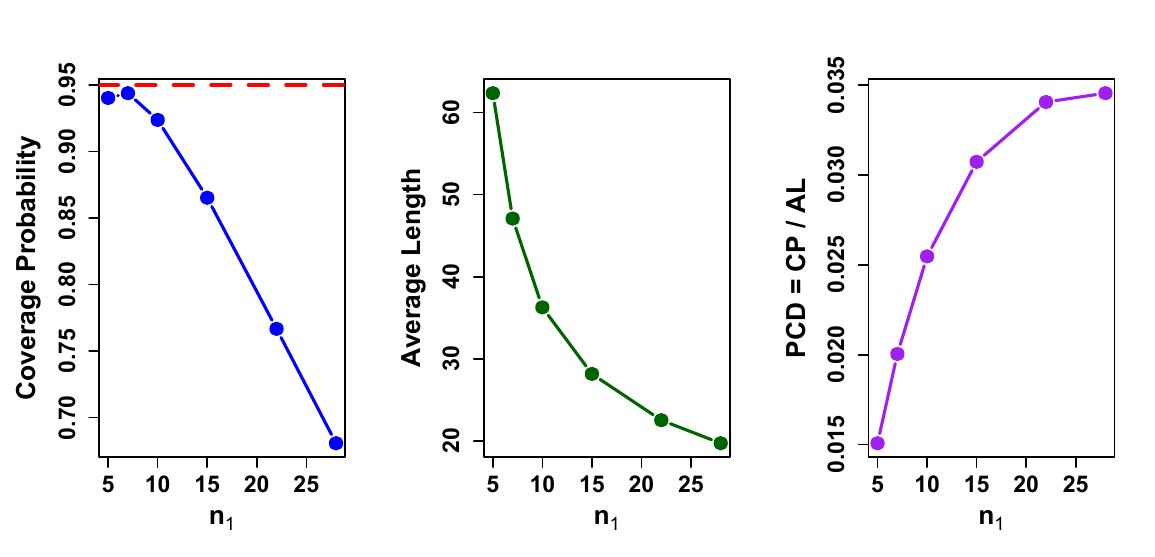}
		\caption{HPD credible interval (Gibbs) for $n_2=8$, $\mu=0$, $\sigma_1=1, \sigma_2=1$}
		\label{p5}
	\end{figure}

	\section{Real data analysis}\label{da}
	In this section, we present a real life data analysis to illustrate the applicability of the derived result for the balanced one-way random effect model.
	%\subsection{Example 1} In pharmaceutical company, the evaluation of assay methods is an important aspect of drug analysis. For a given drug, it is often necessary to estimate the between-day and within-day variance components to assess the reliability and reproducibility of the assay method. A balanced one way random effect can be used to estimate the between and within day variance components. For example, suppose that three assays $(n_2=3)$ are performed on each of three different days $(n_1=3)$ for a particular drug. The assay results are expressed as percentage recovery, defined as the ratio of the measured assay amount to the amount added to the placebo. The resulting assay data (in $\%$ recovery) are as follows:
	%\begin{table}[h!]
	%	\centering
	%	\label{d0}
	%	\begin{tabular}{ccc}
		%		\Xhline{1pt}
		%		Day 1 & Day 2 & Day 3 \\
		%		\hhline{===}
		%		101.91 & 101.86 & 102.18 \\
		%		101.26 & 101.13 & 101.20 \\
		%		100.40 & 101.18 & 101.17 \\
		%		\Xhline{1pt}
		%	\end{tabular}
	%\end{table}
	%The proposed estimated values for $\lambda$ are $\delta_{01}= 1.9596563672$, $\delta_{S1}=1.9596563672$ and  $\delta_{\varphi_2^{*2}}= 1.9596563669 $.
	%\subsection{Example 1}
	We consider the catapult experimental data reported by \cite{schubert1992catapult}. The experiment was designed to study the factors affecting the distance traveled by a ball in a catapult system. Four factors, namely hook position, arm length, start angle, and stop angle, were considered at two levels. An eight-run experimental design was used, and three replicate measurements of the response were obtained for each experimental run. Thus, the dataset consists of 24 observations. The response values reported in Table 1 of \cite{schubert1992catapult} are used for the present analysis.
	
	\begin{table}[h!]
		\centering
		\label{d01}
		\begin{tabular}{c|ccc}
			\Xhline{1pt}
			Test runs&&Replicated values for distance&\\
			r& $Y_1$ & $Y_2$ & $Y_3$ \\
			\hhline{====}
			1& 28&27.1&26.2 \\
			2&46.3&43.5&46.5 \\
			3&21.9&21&20.1 \\
			4&52.9&53.7&52\\
			5&75&73.1&74.3\\
			6&127.7&126.9&128.7\\
			7&86.2&86.5&87\\
			8&195&195.9&195.7\\
			\Xhline{1pt}
		\end{tabular}
	\end{table}
	The data consists 24 observations which come from eight experimental runs, with three replicated measurements for each run. Thus, each group contains the same number of replications, confirming that the experimental design is balanced.  To analyze the between-run and within-run sources of variation, we consider the one-way random-effects model
	$Y_{ij}=\mu+a_i+\epsilon_{ij},~~ i=1,\dots,8,~~ j=1,2,3,$
	where $a_i$ represents the random effect associated with the $i$-th run and $\epsilon_{ij}$ denotes the within-run random error. The Kolmogorov--Smirnov tests for residual normality yielded $p-$values of 0.6548, while Bartlett's tests for homogeneity of variance yielded $p-$values of 0.7445. Thus, the diagnostics provide no significant evidence of departure from the assumed normality and homogeneity conditions, supporting the use of the model for the analysis.
	
	To compare the between-run and within-run variances, we test
	$H_0:\sigma_1^2=\sigma_2^2
	\quad\text{Vs}\quad
	H_1:\sigma_1^2\ne\sigma_2^2.$ The calculated F-statistic is $F^*=2814.69$ with degrees of freedom $(7,16)$, giving a $p$-value of $3.33\times10^{-23}$. Since the $p$-value is less than $0.05$, we reject $H_0$. Thus, the between-run and within-run variances are significantly different.
	
	\begin{table}[h!]
		\centering
		\caption{Estimated values of the improved estimators of $\lambda$}
		\label{ta11}
		\begin{tabular}{c|ccc}
			\Xhline{1pt} 
			&$\delta_{0}$ & $\delta_{S}$ & $\delta_{\varphi_2^*}$ \\
			\hhline{====}
			$L_1(t)$ &6567.6117&6567.6117&6568.7412\\
			$L_2(t)$ &9851.4176&9851.4176&9805.2815\\
			$L_3(t)$ &12461.1671&12461.1671&12326.4804\\
			\Xhline{1pt} 
		\end{tabular}
	\end{table}
	
	\begin{table}[h!]
		\centering
		\caption{Computed confidence intervals with length of $\sigma_r$}
		\label{ta12}
		\begin{tabular}{cccccc}
			\Xhline{1pt} 
			&ACI & Bootstrap-t & Bootstrap-p &HPD &GCI \\
			\hhline{======}
			Lower &-657.843 &1187.009 & 848.528& 1628.145&1129.660\\
			Upper &7224.788 &16642.958 &11839.092 &31764.527 &16681.574\\
			Length &7882.631 & 15455.948 &10990.564 &30136.382 &15551.914\\
			\Xhline{1pt} 
		\end{tabular}
	\end{table}
	
	\section{Conclusion}\label{se7}
	In this paper, we studied the problem of estimating the ratio of variance components $\sigma_r$ in a balanced one-way random-effects model under a general class of scale-invariant loss functions, $L(t)$. Since the ratio $\sigma_r$ comes in the well known relation $\lambda =1 + n_2\sigma_r = E(\mbox{MSA})/E(\mbox{MSE})$. So estimating $\lambda$ is equivalent to estimating $\sigma_r$. We first characterized the class of $\mathcal{G}$-equivariant estimators of $\lambda$ and derived the BEE. We consider a broader class of estimators and, using Stein's argument, obtained an improved estimator $\delta_{S1}$ that dominates the BEE. Further, we used the IERD method of Kubokawa to obtain a smooth, improved estimator $\delta_{\phi_2^*}$ that dominates the BEE. After computing a general form of the improved estimator for $\lambda$, we compute the results for three special loss functions: quadratic loss $L_1(t)$, entropy loss $L_2(t)$, and symmetric loss $L_3(t)$.

	In addition to point estimation, we derived several interval estimators for $\sigma_r$, namely the MLE-based asymptotic confidence interval, two parametric bootstrap intervals (bootstrap-p and bootstrap-t), a generalized confidence interval based on the generalized pivotal variable approach and an HPD credible interval obtained via Gibbs sampling under Jeffreys' non-informative prior.
	
	We have conducted an extensive simulation study to compare the risk performance of the proposed improved point estimator with that of BEE across all three loss functions. For interval estimation, the simulation study compared the proposed methods using the criteria CP, AL and PCD. We can say from the simulation study that no single method uniformly dominates the others across all criteria, CP, AL and PCD. Finally, we have done a real-life data analysis to illustrate the practical applicability of both the proposed point estimators and the interval procedures.
	
	Overall, the results show that using an equivariance structure and under a general loss function gives an improvement over the standard ANOVA/MLE-based estimator of $\sigma_r$ for point estimation. In future, the work could extend this approach to unbalanced one-way random effects models, higher-way random and mixed effects models, and non-normal error structures, as well as develop Bayesian estimators under other classes of priors.

	\section*{Data availability statement}
	Data used in this study are publicly available in the literature and cited in this manuscript.
	\section*{Disclosure statement}
	No potential conflict of interest was reported by the authors.
	\section*{Funding}
	Lakshmi Kanta Patra thanks the Anusandhan National Research Foundation, India for providing financial support to carry out this research with project number MTR/2023/000229.

	\end{document}